\documentclass[3p,times,11pt,leqno]{elsarticle}

\usepackage{microtype}
\usepackage[english]{babel}

\usepackage{stmaryrd}
\usepackage{latexsym}
\usepackage{alltt}
\usepackage{amssymb}
\usepackage{amsmath}
\usepackage{fleqn}
\usepackage[all]{xy}
\SelectTips{eu}{}
\usepackage{xcolor}

\usepackage{pdfsync}
\usepackage[active]{srcltx}

\usepackage[colorlinks]{hyperref}

\usepackage{amsmath,amssymb}

\ifluatex

\usepackage{newtxtext}
\usepackage{tgtermes}
\usepackage[T1]{fontenc}
\usepackage{textcomp}
\usepackage{newtxmath}

\usepackage{textcase}

\fi

\usepackage{amsthm}

\usepackage{tikz}
\usetikzlibrary{positioning}
\usetikzlibrary{calc}
\usetikzlibrary{arrows}
\usetikzlibrary{intersections}
\usetikzlibrary{backgrounds}

\tikzstyle{reverseclip}=[insert path={(-99cm, -99cm) rectangle (99cm, 99cm)}]
\tikzstyle{framed}=[draw,fill=white,drop shadow,inner sep=1.3ex,outer sep=5pt]

\usepackage[inline]{enumitem}
\setlist[itemize]{
  topsep=.5em,
  parsep=2pt,
  itemsep=.5\parskip,
  labelsep=.5em
}

\setlist[enumerate,1]{
  topsep=.5em,
  parsep=2pt,
  itemsep=.5\parskip,
  labelsep=.5em
}

\setlist[description]{
  font=\normalfont\itshape
}

\raggedright 

\newtheorem{definition}{Definition}[section]
\newtheorem{example}{Example}[section]

\newtheorem{proposition}{Proposition}[section]
\newtheorem{fact}{Fact}[section]
\newtheorem{lemma}{Lemma}[section]

\newtheorem{corollary}{Corollary}[section]

\newcommand\pto{\mathrel{\ooalign{\hfil$\mapstochar$\hfil\cr$\to$\cr}}}

\newcommand{\co}{\,\colon\;}
\newcommand{\ra}{\rightarrow}
\newcommand{\Ra}{\Rightarrow}

\newcommand{\st}{\mathrm{st}}
\newcommand{\op}{\mathrm{op}}

\newcommand{\T}{\mathcal{T}\!\!}
\newcommand{\impl}{\Rightarrow}

\newcommand{\vsp}{\vspace{.5em}}

\def\C{\boldsymbol{C}}
\def\P{\mathcal{P}}
\def\I{\mathcal{I}}
\newcommand{\Mod}{\mathit{Mod}}
\newcommand{\Sen}{\mathit{Sen}}
\newcommand{\Sign}{\mathit{Sign}}

\newcommand{\Th}{\mathit{Th}}
\newcommand{\thh}{\mathit{t}}
\def\PoSet{\boldsymbol{PoSET}}
\def\Set{\boldsymbol{SET}}
\def\Setp{\boldsymbol{P\!f\!n}}
\def\CAT{\boldsymbol{C\!AT}}
\def\3/2{\frac{3}{2}}
\newcommand{\dom}[1]{\Box{#1}}
\newcommand{\cod}[1]{{#1}\Box}
\def\DOM{\mathrm{dom}}

\def\PL{\mathcal{P\!\!L}}
\def\MSA{\mathcal{M\!S\!A}}

\usepackage{array}
\usepackage{arydshln}

\usepackage{setspace}

\DeclareSymbolFont{ams}{U}{msa}{m}{n}
\DeclareSymbolFontAlphabet{\mathams}{ams}

\begin{document}


\title{$\3/2$-Institutions: an institution theory for
  conceptual blending}

\author{R\u{a}zvan Diaconescu}
\ead{Razvan.Diaconescu@imar.ro}
\address{Simion Stoilow Institute of Mathematics of the Romanian Academy}
\date{\today}

\begin{abstract}
\noindent
We develop an extension of institution theory that accommodates
implicitly the partiality of the signature morphisms and its
syntactic and semantic effects.
This is driven primarily by applications to conceptual blending,
but other application domains are possible (such as software
evolution). 
The particularity of this extension is a reliance on ordered-enriched
categorical structures.  
\end{abstract}

\maketitle
\section{Introduction}

\subsection{Institution theory}

The mathematical context of our work is the theory of institutions
\cite{Goguen-Burstall:Institutions-1992} which is a three-decades-old
category-theoretic abstract model theory that traditionally has been
playing a crucial foundational role in formal
specification(e.g. \cite{sannella-tarlecki-book}).   
It has been introduced in
\cite{Goguen-Burstall:Introducing-institutions-1983} as an answer to
the explosion in the number of population of logical systems there, as 
a very general mathematical study of formal logical systems, with
emphasis on semantics (model theory), that is not committed to any
particular logical system.
Its role has gradually expanded to other areas of logic-based computer
science, most notably to declarative programming and ontologies. In
parallel, and often in interdependence to its role in computer
science, in the past fifteen years it has made important contributions
to model theory through the new area called
\emph{institution-independent model theory} \cite{iimt} -- an abstract
approach to model theory that is liberated from any commitment to
particular logical systems.  
Institutions thus allowed for a smooth, systematic, and uniform
development of model theories for unconventional logical systems, as
well as of logic-by-translation techniques and of heterogeneous
multi-logic frameworks.  

Mathematically, institution theory is based upon a category-theoretic
\cite{MacLane:Categories-for-the-Working-Mathematician-1998}
formalization of the concept of logical system that includes the
syntax, the semantics, and the satisfaction relation between them. 
As a form of abstract model theory, it is the only one that
treats all these components of a logical system fully abstractly. 
In a nutshell, the above-mentioned formalization is a
category-theoretic structure \((\Sign, \Sen, \Mod, \models)\), called
\emph{institution}, that consists of 
\begin{enumerate*}[label=(\textit{\alph*})]
  
\item a category \(\Sign\) of so-called \emph{signatures},
  
\item two functors, \(\Sen \colon \Sign \to \Set\) for the syntax,
  given by sets of so-called \emph{sentences}, and \(\Mod \colon
  \Sign^{\varominus} \to \CAT\) for the semantics, given by categories of
  so-called \emph{models}, and 
  
\item for each signature \(\Sigma\), a binary \emph{satisfaction
    relation} \(\models_{\Sigma}\) between the \(\Sigma\)-models,
  i.e.\ objects of \(\Mod(\Sigma)\), and the \(\Sigma\)-sentences,
  i.e.\ elements of \(\Sen(\Sigma)\), 

\end{enumerate*}
such that for each morphism \(\varphi \colon \Sigma \to \Sigma'\) in
the category \(\Sign\), each \(\Sigma'\)-model \(M'\), and each
\(\Sigma\)-sentence \(\rho\) the following \emph{Satisfaction
  Condition} holds: 
\[
  M' \models_{\Sigma'} \Sen(\varphi)(\rho)
  \qquad \text{if and only if} \qquad
  \Mod(\varphi)(M') \models_{\Sigma} \rho.
\]

\vspace{-\parskip}
Because of its very high level of abstraction, this definition
accommodates not only well established logical systems but also very
unconventional ones. Moreover, it has served and it may serve as a
template for defining new ones. 
Institution theory approaches logic and model theory from a
relativistic, non-substantialist perspective, quite different from the 
common reading of formal logic. 
This does not mean that institution theory is opposed to the
established logic tradition, since it rather includes it from a higher
abstraction level. 
In fact, the real difference may occur at the level of the development
methodology: top-down in the case of institution theory, versus
bottom-up in the case of traditional logic. 
Consequently, in institution theory, concepts come naturally as
presumed features that a logical system might exhibit or not, and are
defined at the most appropriate level of abstraction; in developing
results, hypotheses are kept as general as possible and introduced on
a by-need basis. 

\subsection{Conceptual blending}

Our work constitutes an effort to provide adequate mathematical
foundations to \emph{conceptual blending}, which is an important research
problem in the area of \emph{computational creativity}.
This is a relatively recent multidisciplinary science, with
contributions from/to artificial intelligence, cognitive sciences,
philosophy and arts, going back at least until to the notion of
\emph{bisociation}, presented by Arthur Koestler
\cite{Koestler:The-act-of-creation-1964}.   
Its aims are not only to construct a program that is capable of
human-level creativity, but also to achieve a better understanding and
to provide better support for it. 
Conceptual blending was proposed by Fauconnier and Turner
\cite{Fauconnier-Turner:Conceptual-Integration-Networks-1998} as a
fundamental cognitive operation of language and common-sense, modelled
as a process by which humans subconsciously combine particular
elements of two possibly conceptually distinct notions, as well as
their relations, into a unified concept in which new elements and
relations emerge. 





  
The structural aspects of this cognitive theory have been given
rigorous mathematical grounds by Goguen
\cite{Goguen:Algebraic-Semiotics-1999,Goguen:What-Is-a-Concept-2005},
based upon category theory. 
In this formal model, concepts are represented as logical theories
giving their axiomatization. 
Goguen used the algebraic specification language OBJ
\cite{Goguen-Winkler-Meseguer-Futatsugi-Jouannaud:Introducing-OBJ-2000}
to axiomatize the concepts, a language that is based upon a refined
version of equational logic; but in fact the approach is independent
of the logical formalism used (this is why category theory is
involved).  
This approach is illustrated by the diagram in
Figure~\ref{figure:blending}, which has to be read in an \emph{order-enriched 
categorical} context:  
The nodes correspond to logical theories and the arrows to theory
morphisms, but the diagram does not commute in a strict sense. 
There is only a \emph{lax} form of commutativity, meaning that the
compositions in the left- and the right-hand sides of the diagram are
both `less' than the arrow at the centre. 
The `less' comes from the fact that the arrows (to be interpreted as
theory morphisms) are subject to an ordering that reflects the fact
that they correspond to \emph{partial} rather than total mappings. 

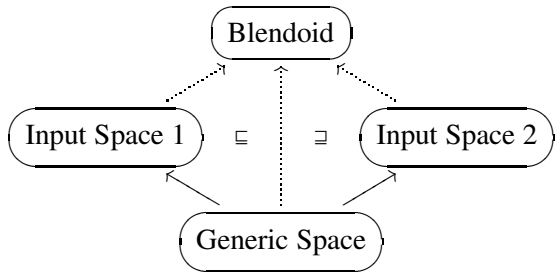
\begin{figure}
\begin{center}
  \[
    \xymatrix@R=8ex@C=6em@!0{
      &*++[F-:<13pt>]+{\text{Blendoid}}&\\
      *++[F-:<13pt>]+{\text{Input Space 1}}\ar@{.>}[ur] \ar @{} [r] |(.775){\sqsubseteq}& & *++[F-:<13pt>]+{\text{Input Space 2}}\ar@{.>}[ul] \ar @{} [l] |(.775){\sqsupseteq}\\
      &  *++[F-:<13pt>]+{\text{Generic Space}}\ar[ul]\ar[ur]\ar@{.>}[uu] &
    }
  \]
\end{center}


  
  \caption{3/2-categorical blending}
  \label{figure:blending}
\end{figure}

In the above-mentioned work by Goguen there are convincing arguments,
supported by examples, for this partiality aspect, which represents
very much a departure to a different mathematical realm than that of
logical theories (even when considered in a very general sense, as
commonly done in modern computer science).  
In category-theoretic terms, this means that we need to consider there
categories equipped with partial orders on the hom-sets that are
preserved by the compositions of arrows/morphisms. 
These are special instances of 2-categories (a rather notorious
concept), somehow half-way between ordinary categories and
2-categories; according to Goguen, this is what motivates the term
\emph{$\3/2$-category}. 
To summarise the main mathematical idea underlying theory blending as
it stands now: 
\begin{quotation}
  \noindent\itshape
  Theory blending is a cocone in a $\3/2$-category in which objects
  represent logical theories and arrows correspond to partial mappings
  between logical theories.  
\end{quotation} 
There is still a great deal of thinking on whether the cocone should
actually be a colimit (in other words, a minimal cocone) or not
necessarily.   
An understanding of this issue is that blending should not  
necessarily be thought as a colimit, but that colimits are related to
a kind of \emph{optimality} principle.  
Moreover, since $\3/2$-category theory has several different concepts
of colimits, there is still thinking about which of those is most
appropriate for modelling the blending operation. 

Goguen's ideas about theory blending benefited from an important boost
with the European FP7 project COINVENT
\cite{Schorlemmer-Smaill-Kuhnberger-Kutz-Colton-Cambouropoulos-Pease:COINVENT-2014}
that has adopted them as its foundations.  
Based on this, a creative computational system has been implemented
and demonstrated in fields like mathematics
\cite{GomezRamirez:Fields-and-Galois-Theory-through-Formal-Conceptual-Blending-MSP}
and music
\cite{Eppe-Confalonieri-Maclean-KaliakatsosPapakostas-Cambouropoulos-Schorlemmer-Codescu-Kuhnberger:Computational-Invention-of-Cadences-and-Chord-Progressions-2015}
(although both use the strict rather than the $\3/2$-version of
category theory).    

\subsection{$\3/2$-institutions}

However, the COINVENT approach still lacks crucial theoretical
features, especially a proper semantic dimension. 
Such a dimension is absolutely necessary when talking about concepts
because meaning and interpretation are central to the idea of
concept. 
For example, the idea of consistency of a concept depends on the semantics.
If one considers also the abstraction level of Goguen's approach in
its general form, of non-commitment to particular logical systems,
then \emph{the institution-theoretic dimension appears as inevitable}.   
In fact, Goguen argued for the role of institution theory in 
\cite{Goguen:Mathematical-Models-of-Cognitive-Space-and-Time-2006},
and so does the COINVENT project.  
However, institution theory cannot be used as such in a proper way
because, as it stands now, it cannot capture the partiality of theory
morphisms (which boils down to the partiality of signature
morphisms). 
Although the treatment of signatures and their morphisms as an
abstract category \(\Sign\) seems to do this, the implications of this
partiality go beyond the common concept of institution.
The the sentence translations \(\Sen(\varphi)\) ought to be allowed to
be partial rather than total functions, and that the model reducts
\(\Mod(\varphi)\) ought to be allowed to map models to \emph{sets of
  models} rather than single models. 

Therefore we define a $\3/2$-categorical extension of the concept of
institution, called  \emph{$\3/2$-institution}, that accommodates
those aspects and that starts from an abstract $\3/2$-category of 
signatures. 
Moreover, based on this, we unfold a theory of $\3/2$-institutions
aimed as a general institution theoretic foundations for conceptual
blending. 

\subsection{Other applications: the problem of merging software changes}

The diagram in Figure~\ref{figure:blending} that depicts the process
of theory blending also has an important interpretation in software
engineering: 
In large software-development projects, it often happens that a part
of the system is being modified (deleting of code also allowed) by
several different programmers concurrently, after which it is
necessary to merge the changes to form a single consistent version.  
Even cooperative distributed writing of papers or documents may fall
under this topic; writing scientific papers in \LaTeX\ certainly
qualifies, as \LaTeX\ is indeed a programming language. 
Like in the case of theory blending, a $\3/2$-categorical approach
is necessary (changes being modelled as partial mappings) 
\cite{Goguen:Categorical-Approaches-to-Merging-Software-Changes-1995}
but this is not enough because of not being able capture the semantic 
dimension of software.   
For example in order to be able to have a notion of consistency for
merges we need to enhance the approach with a model theory. 
This software engineering problem is a second application domain that
drives our development of the theory of $\3/2$-institutions.

\subsection{Contributions and Structure of the Paper}

The paper is structured as follows:
\begin{enumerate}

\item In a preliminary section we introduce some basic category
  theoretic notations and terminology, with emphasis on
  $\3/2$-categories. 

\item In a section on $\3/2$-institutions we start by recalling the
  basic concepts of (ordinary) institution theory, then we refine this
  to the concept of $\3/2$-institution, provide a collection of relevant
  examples, and develop basic $\3/2$-institution theoretic concepts
  and results on:
  \begin{itemize}

  \item \emph{$\3/2$-institutional seeds}, that constitute a simple abstract
    scheme that underlies the definition of many
    $\3/2$-institutions of interest and that provides a general
    framework for an easy derivation and understanding of important
    $\3/2$-institutional properties. 

  \item Theory morphisms, that parallels the corresponding concept
    from ordinary institution theory but only to a limited extent,
    since $\3/2$-institution theory admits several relevant concepts
    of theory morphisms. 

  \item Model amalgamation, that extends the corresponding concept
    from ordinary institution theory to $\3/2$-institutions. 

  \end{itemize}

\item We dedicate a special section to the presentation of a scheme
  for approaching conceptual blending with $\3/2$-institutions that
  essentially replaces the currently prevalent idea of looking for
  colimits of theories with another idea, of looking for lax cocones
  with model amalgamation. 
  Our scheme is supported by the mathematical results of the previous
  sections, and in addition to that it has also a number of parameters
  that makes it quite flexible in the applications. 

\end{enumerate}

\section{Category-theoretic and other preliminaries}

\subsection{Categories, monads}

In general we stick to the established category theoretic terminology
and notations, such as in 
\cite{MacLane:Categories-for-the-Working-Mathematician-1998}.
But unlike there we prefer to use the diagrammatic notation for
compositions of arrows in categories, i.e. if $f \co A \ra B$ and
$g\co B \ra C$ are arrows then $f;g$ denotes their composition. 
The domain of an arrow/morphism $f$ is denoted by $\dom{f}$ while its
codomain is denoted by $\cod{f}$. 
$\Set$ denotes the category of sets and functions and $\CAT$ the
``quasi-category'' of categories and functors.\footnote{This means it
  is bigger than a category since the hom-sets are classes rather than
  sets.}  

The \emph{dual} of a category $\C$ (obtained by formally reversing its
arrows) is denoted by $\C^\varominus$. 

Given a category $\C$, a triple $(\Delta,\delta,\mu)$ constitutes a
\emph{monad} in $\C$ when $\Delta \co \C \to \C$, and $\delta$ and
$\mu$ are natural transformations $\Delta^2 \Ra \Delta$ and $1_{\C}
\Ra \Delta$, respectively such that following diagrams commute: 
$$\xy
\xymatrix{
\Delta(\Sigma) \ar[r]^{\delta_{\Delta(\Sigma)}}
\ar[dr]_{1_{\Delta(\Sigma)}} & \Delta^2 (\Sigma) \ar[d]^{\mu_\Sigma} &
\ar[l]_{\Delta(\delta_\Sigma)} \ar[dl]^{1_{\Delta(\Sigma)}} 
\Delta(\Sigma) & \Delta^3(\Sigma) \ar[r]^{\mu_{\Delta(\Sigma)}}
  \ar[d]_{\Delta(\mu_\Sigma)} & \Delta^2 (\Sigma) \ar[d]^{\mu_\Sigma} \\  
 & \Delta(\Sigma) & & \Delta^2 (\Sigma) \ar[r]_{\mu_\Sigma} &
 \Delta(\Sigma) 
}
\endxy$$
The \emph{Kleisli category} $\C_\Delta$ of the monad
$(\Delta,\delta,\mu)$ has the objects of $\C$ but an arrow
$\theta_\Delta \co A \to B$
in $\C_\Delta$ is an arrow $\theta \co A \to \Delta(B)$ in $\C$.  
The composition in $\C_\Delta$ is defined as shown below: 
$$\xy
\xymatrix{
A \ar[d]_{\theta_\Delta} & A \ar[d]^\theta & \\
B \ar[d]_{\theta'_\Delta} & \Delta(B) \ar[d]^{\Delta(\theta')} & \\
C & \Delta^2 (C) \ar[r]_{\mu_{C}} & \Delta(C)
}
\endxy$$

The following functor extends the well known power-set functor from
sets to categories:  

\begin{definition}\label{power-cat-dfn}
The \emph{power-set functor on categories} $\P \co \CAT \ra \CAT$ is
defined as follows:  
\begin{itemize}

\item for any category $\C$,
\begin{itemize}

\item $|\P\C| = \{ A \mid A \subseteq |C| \}$ and 
$\P \C(A,B) = 
\{ H \subseteq \C \mid \dom{h} \in A, \cod{h} \in B \text{ for each }
h \in H \}$; and

\item composition is defined by 
$H_1 ; H_2 = \{ h_1 ; h_2 \mid h_1 \in H_1, h_2 \in H_2, \cod{h_1} =
\dom{h_2}  \}$;  then $1_A = \{ 1_a \mid a\in A \}$ are the identities.

\end{itemize}

\item for any functor $F \co \C \ra \C'$,  
$\P F(A) = F(A) \subseteq |\C'|$ and $\P F(H) = F(H) \subseteq \C'$.

\end{itemize}
Moroever, like in the case of sets, this construction extends to a
\emph{monad $(\P,\{\_\},\cup)$} in $\CAT$. 
Then $\CAT_{\!\!\P}$ denotes its associated Keisli category. 
\end{definition}

\subsection{Partial functions}

A \emph{partial function} $f \co A \pto B$ is a binary relation 
$f \subseteq A \times B$ such that $(a,b), (a,b') \in f$ implies 
$b = b'$. 
The \emph{definition domain} of $f$, denoted $\DOM(f)$ is the set 
$\{ a \in A \mid \exists b \ (a,b)\in f \}$. 
A partial function $f\co A \pto B$ is called \emph{total} when
$\DOM(f) = A$.  
We denote by $f^0$ the restriction of $f$ to $\DOM(f) \times B$; this
is a total function. 
Partial functions yield a subcategory of the category of binary
relations, denoted $\Setp$. 
If $A'\subseteq A$ by $f(A')$ we denote the set 
$\{ b \mid \exists a\in A', (a,b)\in f \}$.
It is easy to check the following (though not as immediate as in the
case of the total functions): 

\begin{lemma}\label{pfun-lem}
Given partial functions $f \co A \pto B$ and $g \co B \pto C$ and 
$A' \subseteq A$ we have that $(f;g)(A') = g(f(A'))$. 
\end{lemma}

\subsection{$\3/2$-categories}

A \emph{$\3/2$-category} is just a category such that its
hom-sets are partial orders, and the composition preserve these
partial orders.  
In the literature $\3/2$-categories are also called \emph{ordered
  categories} or \emph{locally ordered categories}. 
In terms of enriched category theory \cite{kelly-cat}, $\3/2$-category
are just categories enriched by the monoidal category of partially
ordered sets.
 
Given a $\3/2$-category $\C$ by $\C^\varobar$ we denote its `vertical'
dual which reverses the partial orders, and by $\C^\varoplus$ its
double dual $\C^{\varominus\varobar}$. 
Given $\3/2$-categories $\C$ and $\C'$, a \emph{strict $\3/2$-functor} 
$F \co \C \ra \C'$ is a functor $\C \ra \C'$ that preserves the
partial orders on the hom-sets.
\emph{Lax functors} relax the functoriality conditions  
$F(h);F(h') = F(h;h')$ to  $F(h);F(h') \leq F(h;h')$ 
(when $\cod{h} = \dom{h'}$) and $F(1_A) = 1_{F(A)}$  to $1_{F(A)} \leq
F(1_A)$.  
If these inequalities are reversed then $F$ is an 
\emph{oplax functor}. 
This terminology complies to \cite{borceux94} and to more recent
literature, but in earlier literature \cite{kelly-street74,jay91} this
is reversed.   
Note that oplax + lax = strict. 
In what follows whenever we say ``$\3/2$-functor'' without the
qualification ``lax'' or ``oplax'' we mean a functor which is either
lax or oplax.  

Lax functors can be composed like ordinary functors; we denote by
$\3/2 \CAT$ the category of $\3/2$-categories and lax functors. 
  
Most typical examples of a $\3/2$-category are $\Setp$ -- the category
of partial functions in which the ordering between partial functions
$A \pto B$ is given by the inclusion relation on the binary relations
$A \ra B$, and $\PoSet$ -- the category partial ordered sets (with
monotonic mappings as arrows) with orderings between monotonic
functions beign defined point-wise ($f \leq g$ if and only if $f(p)
\leq g(p)$ for all $p$).   

\begin{definition}
Let us consider the power-set monad on categories of
Dfn.~\ref{power-cat-dfn}. Given the partial order on each $\P \C$
given by category inclusions, the Kleisli category $\CAT_{\!\!\P}$ admits a
two-fold refinement to a $\3/2$-category:
\begin{enumerate}

\item morphisms $\C \ra \P \C'$ are allowed to be lax functors rather
  than (strict) functors, and  

\item we consider the point-wise partial order on the class of the lax
  functors  $\C \ra \P \C'$ that is induced by the partial order on
  $\P \C'$.  

\end{enumerate}
Let us denote the $\3/2$-category thus obtained by $\3/2 (\CAT_{\!\!\P})$. 
\end{definition}

Unlike in the case of ordinary categories, colimits in
$\3/2$-categories come in several different flavours according to the
role played by the order on the arrows. 
Here we recall some of these for the particular emblematic case of
pushouts; the extension to other types of colimits being obvious.  

Given a span $\varphi_1,\varphi_2$ of arrows in a $\3/2$-category, a
\emph{lax cocone} for the span consists of arrows
$\theta_0,\theta_1,\theta_2$ such that there are inequalities as shown in
the following diagram:
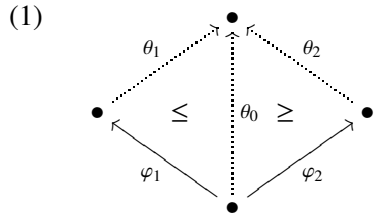
\begin{equation}\label{lax-cocone-equation}
\xymatrix @C-1.5em{
 & & \bullet & & \\
\bullet \ar@{.>}[urr]^{\theta_1} & { \ \ \ \ \leq} & &
{ \geq \ \ \ \ } &
\bullet \ar@{.>}[ull]_{\theta_2}\\
 &  & \bullet \ar[ull]^{\varphi_1} \ar[urr]_{\varphi_2} \ar@{.>}[uu]_{\theta_0}& &
}
\end{equation}
When the two inequalities are both equalities, this is a \emph{strict}
cocone. 
In this case $\theta_0$ is redundant and the data collapses to the
equality $\varphi_1 ; \theta_1 = \varphi_2 ; \theta_2$. 

A lax cocone like in diagram \eqref{lax-cocone-equation} is:
\begin{itemize}

\item \emph{pushout} when it is strict and for any strict cocone
  $\theta'_1, \theta'_2$ there exists and unique arrow $\mu$ that is
  mediating, i.e. $\theta_k ; \mu = \theta'_k$, $k = 1,2$; 

\item \emph{lax pushout} when for any lax cocone   
$\theta'_0, \theta'_1, \theta'_2$ there exists an unique mediating
arrow $\mu$, i.e. $\theta_k ; \mu = \theta'_k$, $k = 0,1,2$;

\item \emph{weak (lax) pushout} when the uniqueness condition on the
  mediating arrow is dropped from the above properties; 

\item \emph{near pushout} when for any lax cocone $\theta'_0,
  \theta'_1, \theta'_2$ the set of mediating arrows $\{ \mu \mid
  \theta_k ; \mu \leq \theta'_k, k = 0,1,2 \}$ has a maximal element. 

\end{itemize}
Pushouts are not a proper $\3/2$-categorical concept because they do
not involve in any way the orders on the arrows. \vsp

Lax pushouts represents the instance of a natural concept of colimit
from general enriched category theory \cite{kelly-cat} to
$\3/2$-categories; however in concrete situations, unlike their
cousins from ordinary category theory, they can be very
difficult to grasp and sometimes appearing quite inadequate.
For example in $\Setp$, if 
$\DOM \varphi_1 \cap \DOM \varphi_2 \not= \emptyset$ then the span
$(\varphi_1,\varphi_2)$ does not have a lax pushout. 
This is caused by the discrepancy between a lot of laxity at the
level of diagrams and of the arrows on the one hand (allowing for
unbalanced cocones in which low components may coexist with high
components), and the strictness required in the universal property on
the other hand.  
A remedy for this would be to restrict the cocones to designated
subclasses of arrows as follows.

\begin{definition}[$\T$-colimits]
Given a (1-)subcategory $\T \subseteq \C$ of a $\3/2$-category $\C$, a
\emph{lax $\T$-cocone} for a span $(\varphi_1,\varphi_2)$ is a lax
cocone $(\theta_0,\theta_1,\theta_2)$ for the span such that
$\theta_k\in \T$, $k=0,1,2$.
A \emph{lax $\T$-pushout} is a minimal lax $\T$-cocone, i.e. for any
lax $\T$-cocone  $(\theta'_0,\theta'_1,\theta'_2)$ there exists an unique
mediating arrow $\mu\in \T$ such that $\theta_k ; \mu = \theta'_k$,
$k=0,1,2$. 

This definition extends in the obvious way to general colimits and to
the weak case (by dropping off the requirement on the uniqueness of
$\mu$).  
\end{definition}
For example, in $\Setp$ by letting $\T$ \ be the class of total
functions, \emph{any} span of partial functions admits a lax
$\T$-pushout.  \vsp

Near pushouts (terminology from \cite{jay91}) are much easier to grasp
than lax pushouts (for example in $\Setp$ they are the epimorphic
cocones) but nevertheless they have received only little consideration
due to  their pathology of lacking uniqueness, a property that is
considered crucial for any kind of colimits.   
However in \cite{jay91} it is argued that they constitute a more
proper concept of colimit in a ordered categorical context because it
involves only inequalities and moreover Goguen argues
\cite{Goguen:Algebraic-Semiotics-1999} that their lack
of the uniqueness property is exactly what makes them useful for
modelling conceptual blending; there he calls them  
\emph{$\3/2$-pushouts}.  

\section{$\3/2$-institutions}\label{3/2-institution-sec}

The outline of this section is as follows. 
\begin{enumerate}

\item We recall the concept of \emph{institution} and provide a couple
  of emblematic examples. Some basic institution theoretic concepts
  are alo recalled.  

\item We introduce the definition of \emph{$\3/2$-institutions}.

\item We provide some relevant \emph{examples} of $\3/2$-institutions
  that constitute extensions of well known corresponding institutions
  that accomodate partiality of the signature morphisms. 

\item We introduce the concept of \emph{$\3/2$-institutional seed}
  that serves as a very general way to define $\3/2$-institutions. 
  This is also mathematically convenient especially within the context
  of the study of model amalgamation properties. 

\item We extend the crucial concept of \emph{model amalgamation} from
  common institution theory to $\3/2$-institution theory, and we give
  some general and yet pragmatic sufficient conditions for 
  $\3/2$-institution theoretic model amalgamation. 

\item We extend the concept of \emph{theory morphism} from common
  institution theory to $\3/2$-institutions, what happens being an
  unfolding of the original concept to several concepts of theory
  morphisms.  
  We establish the relationships between these, and we study their
  basic compositionality and model amalgamation properties.  

\item Finally, we introduce and study \emph{theory changes}, which
  represent a different kind of mapping or relationship between
  theories that is relevant especially in foundational studies for the
  problem of merging software changes.  

\end{enumerate}

\subsection{Institutions}

An  \emph{institution} $\I = 
(\Sign^{\I}, \Sen^{\I}, \Mod^{\I}, \models^{\I})$ consists of 
\begin{itemize}

\item a category $\Sign^{\I}$ whose objects are called
  \emph{signatures},

\item a sentence functor $\Sen^{\I} \co \Sign^{\I} \ra \Set$
  defining for each signature a set whose elements are called
  \emph{sentences} over that signature and defining for each signature
  morphism a \emph{sentence translation} function, 

\item a model functor $\Mod^{\I} \co (\Sign^{\I})^{\varominus} \ra \CAT$
  defining for each signature $\Sigma$ the category
  $\Mod^{\I}(\Sigma)$ of \emph{$\Sigma$-models} and $\Sigma$-model
  homomorphisms, and for each signature morphism $\varphi$ the
  \emph{reduct} functor $\Mod^{\I}(\varphi)$,  

\item for every signature $\Sigma$, a binary 
  \emph{$\Sigma$-satisfaction relation}
  $\models^{\I}_{\Sigma} \subseteq |\Mod^{\I} (\Sigma)|
  \times \Sen^{\I} (\Sigma)$, 

\end{itemize}
such that for each morphism 
$\varphi$,  the \emph{Satisfaction Condition}
\begin{equation}
M'\models^{\I}_{\Sigma'} \Sen^{\I}(\varphi)\rho \text{ if and only if  }
\Mod^{\I}(\varphi)M' \models^{\I}_\Sigma \rho
\end{equation}
holds for each $M'\in |\Mod^{\I} (\cod{\varphi})|$ and $\rho \in
\Sen^{\I} (\dom{\varphi})$. 
\[
\xymatrix{
    \dom{\varphi} \ar[d]_{\varphi} & |\Mod^{\I}(\dom{\varphi})|
    \ar@{-}[r]^-{\models^{\I}_{\dom{\varphi}}} & 
    \Sen^{\I}(\dom{\varphi}) \ar[d]^{\Sen^{\I}(\varphi)} \\
    \cod{\varphi} & | \Mod^{\I}(\cod{\varphi})| \ar[u]^{\Mod^{\I}(\varphi)} 
    \ar@{-}[r]_-{\models^{\I}_{\cod{\varphi}}} & \Sen^{\I}(\cod{\varphi})
  }
\] 
We may omit the superscripts or subscripts from the notations of the
components of institutions when there is no risk of ambiguity. 
For example, if the considered institution and signature are clear,
we may denote $\models^{\I}_\Sigma$ just by $\models$. 
For $M = \Mod(\varphi)M'$, we say that $M$ is the
\emph{$\varphi$-reduct} of $M'$.

\begin{example}[Propositional logic -- $\PL$]
\begin{rm}
This is defined as follows.
$\Sign^{\PL} = \Set$, and for any set $P$, $\Sen(P)$ is generated by
the grammar 
\[
S ::= P \mid S \wedge S \mid \neg S 
\]
and $\Mod^{\PL}(P) = (2^P,\subseteq)$.
For any $M \in |\Mod^{\PL} (P)|$, depending on convenience, we may
consider it either as a subset $M\subseteq P$ or equivalently as a
function $M \co P \ra 2 = \{ 0, 1 \}$.

For any function $\varphi \co P \ra P'$, $\Sen^{\PL}(\varphi)$
replaces the each element $p\in P$ that occurs in a sentence $\rho$ by 
$\varphi(p)$, and $\Mod^{\PL}(\varphi)(M') = \varphi;M$ for each
$M'\in 2^{P'}$. 
For any $P$-model $M \subseteq P$ and $\rho\in \Sen^{\PL}(P)$,
$M\models\rho$ is defined by induction on the structure of $\rho$ by
$(M \models p) = (p\in M)$, 
$(M \models \rho_1 \wedge \rho_2) = 
(M \models \rho_1) \wedge (M \models \rho_2)$ and 
$(M \models \neg\rho) = \neg(M \models \rho)$.  
\end{rm}
\end{example}

\begin{example}[Many-sorted algebra -- $\MSA$]
\begin{rm}
  The \emph{$\MSA$-signatures} are pairs $(S,F)$ consisting of a set $S$ of sort
  symbols and of a family $F = \{ F_{w\ra s} \mid w\in S^*, s\in S \}$ of
  sets of function symbols indexed by arities (for the arguments) and sorts (for
  the results).\footnote{By $S^*$ we denote the set of strings of sort symbols.}
  \emph{Signature morphisms} $\varphi \co (S,F) \ra (S',F')$ consist of a
  function $\varphi^\st \co S \ra S'$ and a family of functions 
  $\varphi^\op = 
  \{ \varphi^\op_{w\ra s} \co F_{w\ra s} \ra F'_{\varphi^\st (w) \ra
    \varphi^\st (s)} \mid w\in S^*, s\in S \}$.

  The \emph{$(S,F)$-models} $M$, called algebras, interpret each sort symbol $s$
  as a set $M_s$ and each function symbol $\sigma\in F_{w\ra s}$ as a function
  $M_\sigma$ from the product $M_w$ of the interpretations of the argument sorts
  to the interpretation $M_s$ of the result sort.
  An \emph{$(S,F)$-model homomorphism} $h \co M \ra M'$ is an indexed family of
  functions $\{ h_s \co M_s \ra M'_s \mid s\in S \}$ such that 
  $h_s (M_\sigma (m)) = M'_\sigma (h_w (m))$ for each $\sigma \in F_{w \ra s}$
  and each $m\in M_w$, where $h_w \co M_w \ra M'_w$ is the canonical
  componentwise extension of $h$, i.e.\ 
  $h_w (m_1, \dots, m_n) = (h_{s_1}(m_1), \dots, h_{s_n}(m_n))$ for 
  $w = s_1 \dots s_n$ and $m_i \in M_{s_i}$.

  For each signature morphism $\varphi \co (S, F) \ra (S', F')$, the
  \emph{reduct} $\Mod(\varphi)(M')$ of an $(S', F')$-model $M'$ is defined by
  $\Mod(\varphi)(M')_x = M'_{\varphi(x)}$ for each sort or function symbol $x$
  from the domain signature of $\varphi$.

  For each signature $(S, F)$, $T_{(S,F)} = ((T_{(S,F)})_s )_{s \in S}$
  is the least family of sets such that $\sigma(t) \in (T_{(S,F)})_s$ for all
  $\sigma \in F_{w \ra s}$ and all tuples $t \in (T_{(S,F)})_{w}$.  The elements
  of $(T_{(S,F)})_s$ are called \emph{$(S,F)$-terms of sort $s$}.  For each
  $(S,F)$-algebra $M$, the \emph{evaluation of an $(S,F)$-term $\sigma(t)$ in
    $M$}, denoted $M_{\sigma(t)}$, is defined as $M_\sigma (M_t)$, where $M_t$
  is the componentwise evaluation of the tuple of $(S,F)$-terms $t$ in $M$.

  \emph{Sentences} are the usual first order sentences built from equational
  atoms $t=t'$, with $t$ and $t'$ (well-formed) terms of the same sort, by
  iterative application of Boolean connectives ($\wedge$, $\impl$, $\neg$,
  $\vee$) and quantifiers ($\forall X$, $\exists X$ -- where $X$ is a sorted
  set of variables).
  Sentence translations along signature morphisms just rename the sort and
  function symbols according to the respective signature morphisms.
  They can be formally defined by recursion on the structure of the sentences.
  The satisfaction of sentences by models is the usual Tarskian satisfaction
  defined recursively on the structure of the sentences.  (As a special note for
  the satisfaction of the quantified sentences, defined in this formalisation by
  means of model reducts, we recall that
  $M \models_\Sigma (\forall X)\rho$ if and only if $M' \models_{\Sigma+X} \rho$
  for each expansion $M'$ of $M$ to the signature $\Sigma+X$ that adds the
  variables $X$ as new constants to $\Sigma$.)
\end{rm}
\end{example}

In the following we recall some basic concepts from institution theory
that will play a role in this work. 

For any set $E$ of $\Sigma$-sentences:
\begin{itemize}

\item if $M$ is a any $\Sigma$-model, then by $M \models E$ we
  denote that $M \models e$ for each $e\in E$;

\item $E$ is \emph{consistent} when there
  exists a $\Sigma$-model $M$ such that $M \models E$;

\item if $\rho$ is a $\Sigma$-sentence then $E \models \rho$
  denotes the situation when for each $\Sigma$-model $M$ if $M \models
  E$ then $M \models \rho$ too;

\item by $E^\bullet$ we denote $\{ \rho\in \Sen(\Sigma) \mid E \models
  \rho \}$. 

\end{itemize}

In any institution, a \emph{theory} is a pair $(\Sigma,E)$ consisting
of a signature $\Sigma$ and a set $E$ of $\Sigma$-sentences.  
A \emph{theory morphism} $\varphi \co (\Sigma,E) \ra (\Sigma',E')$ 
is a signature morphism $\varphi \co \Sigma \ra \Sigma'$
such that $E' \models \Sen(\varphi)E$.
It is easy to check that the theory morphisms are closed under the
composition given by the composition of the signature morphisms; this
gives the \emph{category of the theories of $\I$} denoted $\Th^\I$.  
This fact opens the door for the following general construction, that
is quite helpful in several situations, especially in the study of
logic encodings. 

Let $\I = (\Sign,\Sen,\Mod,\models)$ be any institution. 
The \emph{institution of the theories} of $\I$, denoted by 
$\I^{\thh} = (\Sign^{\thh},\Sen^{\thh},\Mod^{\thh},\models^{\thh})$, is
defined by 
\begin{itemize} 

  \item $\Sign^\thh$ is the category $\Th$ of the theories of $\I$, 

  \item $\Sen^\thh(\Sigma,E) = \Sen(\Sigma)$, 

  \item $\Mod^\thh(\Sigma,E)$ is the full subcategory of
    $\Mod(\Sigma)$ determined by those
           models which satisfy $E$, and

  \item for each $(\Sigma,E)$-model $M$ and $\Sigma$-sentence $e$,
    $M \models^\thh_{(\Sigma,E)} e$ if and only if $M \models_\Sigma e$. 
 
\end{itemize}

Model amalgamation properties for institutions formalize the possibility of
amalgamating models of different signatures when they are
consistent on some kind of generalized `intersection' of signatures.
It is one of the most pervasive properties of concrete institutions
and it is used in  a crucial way in many institution theoretic
studies.
A few early examples are 
\cite{sannella-tarlecki88,tarlecki86,jm-granada89,modalg}. 
For the role played by this property in specification theory and in
institutional model theory see \cite{sannella-tarlecki-book} and
\cite{iimt}, respectively.

A \emph{model of a diagram of signature morphisms in an institution}
consists of a model $M_k$ for each signature $\Sigma_k$ in the
diagram such that for each signature morphism $\varphi \co \Sigma_i
\to \Sigma_j$ in the diagram we have that $M_i = \Mod(\varphi)M_j$.    

A commutative square of signature morphisms
$$\xy
\xymatrix { 
\Sigma \ar[r]^{\varphi_1}   \ar[d]_{\varphi_2}
    & \Sigma_1       \ar[d]^{\theta_1} \\
\Sigma_2  \ar[r]_{\theta_2}
    & \Sigma'
}
\endxy$$
is an \emph{amalgamation square} if and only if each model of the span
$(\varphi_1,\varphi_2)$ admits an unique completion to a model of the
square. 
When we drop off the uniqueness requirement we call this a 
\emph{weak model amalgamation square}.

In most of the institutions formalizing conventional or non-conventional 
logics, pushout squares of signature morphisms are model amalgamation
squares  \cite{iimt}.  

In the literature there are several more general concepts of model
amalgamation. 
One of them is model amalgamation for cocones of arbitrary diagrams
(rather than just for spans), another one is model amalgamation for
model homomorphisms. 
Both are very easy to define by mimicking the definitions presented
above. 
While the former generalisation is quite relevant for the intended
applications of our work, the latter is less so since at this moment
model homomorphisms do not seem to play any role in conceptual
blending or in merging of software changes. 
Moreover amalgamation of model homomorphisms is known to play a role
only in some developments in institution-independent model theory 
\cite{iimt}, but even there most involvements of model amalgamation
refers only to amalgamation of models.  

\subsection{$\3/2$-institutions: definition}

\begin{definition}[$\3/2$-institution]\label{ins-dfn}
A \emph{$\3/2$-institution} 
$\I = ( \Sign^{\I}, \Sen^{\I}, \Mod^{\I}, (\models^{\I}_\Sigma)_{\Sigma \in
  |\Sign^{\I}|} )$ consists of 
\begin{itemize}

\item a $\3/2$-category of signatures $\Sign^{\I}$, 

\item an $\3/2$-functor $\Sen^{\I} \co \Sign^{\I} \ra
  \Setp$, called the \emph{sentence functor}, 

\item an lax $\3/2$-functor  
$\Mod^{\I}\co(\Sign^{\I})^{\varoplus}\rightarrow \3/2 (\CAT_{\!\!\P})$, called the
\emph{model functor}, 

\item for each signature $\Sigma\in |\Sign^\I|$ a 
\emph{satisfaction relation} 
$\models_\Sigma^{\I} \ \subseteq \ |\Mod^{\I}(\Sigma)|\times \Sen^{\I}(\Sigma)$ 

\end{itemize}
such that for each morphism 
$\varphi \in \Sign^{\I}$, 
the \emph{Satisfaction Condition}
\begin{equation}\label{sat-cond-eq}
M'\models^{\I}_{\cod{\varphi}} \Sen^{\I}(\varphi)\rho
\ \ \text{ if and only if } \ \
M \models^{\I}_{\dom{\varphi}} \rho 
\end{equation}
holds for each $M'\in |\Mod^{\I} (\cod{\varphi})|$, 
$M \in |\Mod^{\I}(\varphi)M'|$ and 
$\rho \in \DOM (\Sen^{\I} (\varphi))$.
\end{definition}

The difference between $\3/2$-institutions and ordinary institutions,
from now on called \emph{1-institutions}, is determined by the
$\3/2$-categorical structure of the signature morphisms which
propagates to the sentence and to the model functors.   
Consequently the Satisfaction Condition \eqref{sat-cond-eq} takes an
appropriate format. 
Thus, for each signature morphism $\varphi$
its corresponding sentence translation $\Sen(\varphi)$ is a partial
function $\Sen(\dom{\varphi}) \pto \Sen(\cod{\varphi})$ and moreover
whenever $\varphi \leq \theta$ we have that 
$\Sen(\varphi) \subseteq \Sen(\theta)$. 
The sentence functor $\Sen$ can be either lax or oplax; depending on
how is this we may call the respective $\3/2$-institution as
\emph{lax} or \emph{oplax $\3/2$-institution}. 
In many concrete situations it happens that $\Sen$ is strict while
some general results require it to be either lax or oplax or both. 

The model reduct $\Mod(\varphi)$ is an \emph{lax} functor
$\Mod(\cod{\varphi}) \ra \P \Mod(\dom{\varphi})$ meaning that for each
$\Sigma'$-model we have a \emph{set of reducts} rather than a single 
reduct. 
In concrete examples this is a direct consequence of the partiality
of $\varphi$: in the reducts the interpretation of the symbols on
which $\varphi$ is not defined is unconstrained, therefore there may
be many possibilities for their interpretations. 
``Many'' here includes also the case when there is no interpretation.

\begin{definition}
The model functor $\Mod$ \emph{admits emptiness} when there exists a
signature morphism $\varphi$ and a $\cod{\varphi}$-model $M'$ such
that $\Mod(\varphi) = \emptyset$, otherwise it is said that $\Mod$
does not admit emptiness. 
\end{definition}

In examples most often the model functors $\Mod$ do not admit
emptiness, however the general definition does not rule out emptiness
and moreover there are significant examples (we will see in
Sect.~\ref{th-morphism-sec}) when emptiness of $\Mod$ may happen. 

\begin{itemize}

\item[--] 
The fact that $\Mod$ is a $\3/2$-functor implies also that whenever
$\varphi \leq \theta$ we have $\Mod(\theta) \leq \Mod(\varphi)$,
i.e. $\Mod(\theta)M' \subseteq \Mod(\varphi)M'$, etc. 

\item[--]
The lax aspect of $\Mod$ means that for signature morphisms
$\varphi$ and $\varphi'$ such that $\cod{\varphi} = \dom{\varphi'}$ and
for any  $\cod{\varphi'}$-model $M''$, we have that 
\[
\Mod(\varphi)(\Mod(\varphi')M'') \subseteq
\Mod(\varphi;\varphi')M''  
\]
and for each signature $\Sigma$ and for each $\Sigma$-model $M$ that 
\[
M \in \Mod(1_\Sigma)M. 
\]

\item[--]
The lax aspect of the reduct functors $\Mod(\varphi)$ means that for
model homomorphisms $h_1, h_2$ such that $\cod{h_1}=\dom{h_2}$ we have
that 
\[
\Mod(\varphi)(h_1);\Mod(\varphi)(h_2) \subseteq
\Mod(\varphi)(h_1;h_2)
\] 
and for each $M'\in \Mod(\cod{\varphi})$ and each $M\in
\Mod(\varphi)M'$ that 
\[
1_M \in \Mod(\varphi)1_{M'}.
\]
\end{itemize}
As already mentioned above model homomorphisms do not play yet any
role in conceptual blending or in other envisaged applications of
$\3/2$-institutions.
Hence the lax aspect of model functors is for the moment a purely
theoretical feature which is however supported naturally by all
examples. \vsp

In \cite{vidal-tur2010} there is a 2-categorical generalization of the
concept of institution, called \emph{2-institution}, that consider
$\Sign$ to be a 2-category, $\Sen\co \Sign \to \CAT$ and $\Mod\co
\Sign^\varominus \to \CAT$ to be pseudo-functors,
and that takes a (quite sophisticated categorically) many-valued
approach to the satisfaction relation.   
From these we can see immediately that $2$-institutions of
\cite{vidal-tur2010} do not cover the concept of $\3/2$-institution
through the perspective of $\3/2$-categories as special cases of
$2$-categories, the functors $\Sen$ and $\Mod$ in $2$-institutions
diverging from those in $\3/2$-institutions in two ways: they are
pseudo-functors (in $\3/2$-category theory this means just ordinary
functors) and their targets do not match those of
$\3/2$-institutions. 
This lack of convergence is due to the two extensions aiming to
different application domains. 

\begin{definition}[Total signature morphisms]\label{total-dfn}
A signature morphism $\varphi$ in a $\3/2$-institution is 
\begin{itemize}

\item \emph{$\Sen$-maximal} when $\Sen(\varphi)$ is total;

\item \emph{$\Mod$-maximal} when for each $\cod{\varphi}$-model $M'$,
  $\Mod(\varphi)M'$ is a singleton; and 

\item \emph{total} when it is both $\Sen$-maximal and $\Mod$-maximal. 

\end{itemize}
\end{definition}

\begin{corollary}\label{total-cor}
In each $\3/2$-institution the total signature morphisms determine a
1-institution. 
\end{corollary}

\subsection{$\3/2$-institutions: examples}

The following expected example shows that the concept of
$\3/2$-institution constitute a generalisation of the concept of
institution. 

\begin{example}[Institutions]
\begin{rm}
Each 1-institution can be regarded as a $\3/2$-institution that has
all its signature morphisms total (cf. Dfn.~\ref{total-dfn} and
Cor.~\ref{total-cor}).  
\end{rm}
\end{example}

\begin{example}[Propositional logic with partial morphisms of
  signatures -- $\3/2 \PL$]\label{3/2-pl-ex}
\begin{rm}
This example extends the ordinary institution $\PL$ to a
$\3/2$-institution by considering partial functions rather than total
functions as signature morphisms; thus $\Sign = \Setp$. 
\vsp

SENTENCES.
While for each set $P$, $\Sen(P)$ is like in $\PL$, for any partial
function $\varphi \co P \pto P'$ the sentence translation
$\Sen(\varphi)$ translates like in $\PL$ but only the sentences
containing only propositional variables $P$ that are translated
by $\varphi$, i.e. that belong to $\DOM \varphi$; hence the
partiality of $\Sen(\varphi)$.  
More precisely we have that  
$\DOM (\Sen\varphi) = \Sen^{\PL} (\DOM \ \varphi)$ and for each 
$\rho\in \DOM (\Sen\varphi)$ we have that $\Sen(\varphi)\rho =
\Sen^{\PL} (\varphi^0)\rho$ .  
The sentence functor is a \emph{strict} $\3/2$-functor; the main main
part for the functoriality argument for $\Sen$ goes as follows. 
Let $\varphi,\varphi'$ be signature morphisms where $\cod{\varphi} =
\dom{\varphi'}$ and let $\rho\in \Sen(\dom{\varphi})$).
\begin{itemize}

\item 
First we establish the equality of the definition domains:  
$$\begin{aligned}
\DOM \ \Sen(\varphi;\varphi') = \ \ & \Sen^{\PL}(\DOM \ \varphi;\varphi') \\
  = \ \ & \Sen^{\PL}(\{ p \in \DOM \ \varphi \mid \varphi^0 (p) \in \DOM \
      \varphi' \} \\
  = \ \ & \{ \rho\in \Sen^{\PL}(\DOM \ \varphi) \mid
      \Sen^{\PL}(\varphi^0)\rho \in \Sen^{\PL}(\DOM \ \varphi') \} \\
  = \ \ & \{ \rho\in \DOM(\Sen^{\PL}\varphi) \mid
      \Sen^{\PL}(\varphi^0)\rho \in \DOM(\Sen^{\PL}\varphi') \} \\
  = \ \ & \DOM(\Sen\varphi \ ; \ \Sen\varphi').
\end{aligned}$$

\item
The next step is obtained on the basis of the functoriality of
$\Sen^{\PL}$.
For each $\rho \in \DOM \ \Sen(\varphi;\varphi')$ we have: 
\[
\Sen(\varphi;\varphi')\rho = \Sen^{\PL}((\varphi^0;\varphi'^0)\rho =
\Sen^{\PL}(\varphi'^0)(\Sen^{\PL}(\varphi^0)\rho) =
\Sen(\varphi')(\Sen(\varphi)\rho).
\]

\end{itemize}

MODELS.
The $\3/2 \PL$ models and model homomorphisms are those of $\PL$,
but their reducts differ from those in $\PL$. 
Given a partial function $\varphi \co P \pto P'$ and a $P'$-model 
$M' \co P' \ra 2$, 
\[
\Mod(\varphi)M' = \{ M \co P \to 2 \mid M_p =
M'_{\varphi^0(p)} \text{ for all }p\in \DOM \ \varphi \}.
\]
On the model homomorphisms the reduct is defined by 
\[
\Mod(\varphi)(M' \subseteq N') = 
\{ M \subseteq N \mid M \in \Mod(\varphi)M', N \in \Mod(\varphi)N' \}.
\]
The main part of the lax functoriality of $\Mod$ is proved as follows.
Let $\varphi, \varphi'$ be signature morphisms such that $\cod{\varphi}
= \dom{\varphi'}$ and let $M'' \in |\Mod(\cod{\varphi'})|$. 
For any $M\in \Mod(\varphi)(\Mod(\varphi')M'')$ we show that $M \in
\Mod(\varphi;\varphi')M''$. 
Then there exists $M'\in \Mod(\varphi')M''$ such that $M \in
\Mod(\varphi)M'$.  
For any $p\in \DOM(\varphi;\varphi') = $\\
$\{ p\in \DOM \ \varphi \mid \varphi^0 (p) \in \DOM \ \varphi' \}$ we
have that 
$$\begin{array}{rll}
M_p = & M'_{\varphi^0 (p)} & \quad \text{since }p\in \DOM \ \varphi \text{
                             and } M \in \Mod(\varphi)M' \\[.5em]
    = & M''_{\varphi'^0 (\varphi^0 (p))}
      & \quad \text{since }\varphi^0 (p) \in \DOM \ \varphi' \text{
                             and } M' \in \Mod(\varphi')M'' \\[.5em]
    = & M''_{(\varphi;\varphi')^0(p)}. &
\end{array}$$
This shows that $M \in \Mod(\varphi;\varphi')M''$.

Note that $\Mod(1_P)M = \{ M \}$, hence the second condition of the
lax functoriality of $\Mod$ is satisfied in a strict sense.  

The following counterexample shows why $\Mod$ is a proper lax
functor. 
Let $\{ p_1,p_2,p_3 \} \stackrel{\varphi}{\to} \{ p, p_3 \}
\stackrel{\varphi'}{\pto} \{ p_3 \}$ be such that $\varphi(p_1) =
\varphi(p_2) = p'$, $\varphi(p_3) = p_3$ and $\DOM \ \varphi' = \{ p_3
\}$. 
Note that $\DOM (\varphi;\varphi') = \{ p_3 \}$. 
Then we consider any $\cod{\varphi'}$-model $M''$ and $M \in
\Mod(\varphi;\varphi')M''$ such that $M_{p_1} \not= M_{p_2}$. 
Because of the latter condition there is no $M'$ such that $M\in
\Mod(\varphi)M'$. 

Also in general the reduct functors $\Mod(\varphi)$ are proper lax
functors, but this works exactly the other way than in the case of
$\Mod$.    
\begin{itemize}

\item
Let $M' \subseteq N' \subseteq T' \in |\Mod(\cod{\varphi})|$. 
Given $M \subseteq T \in |\Mod(\dom{\varphi})|$ such that $M \in
\Mod(\varphi)M'$ and $T\in \Mod(\varphi)T'$, we may define $N\in
\Mod(\varphi)N'$ by $N_p = N'_{\varphi^0 (p)}$ when $p \in \DOM \
\varphi$ and $N_p = M_p$ otherwise. 
Consequently $M \subseteq N \subseteq T$. 
This shows that we have an equality
\[
\Mod(\varphi)(M' \subseteq N');\Mod(\varphi)(N'
\subseteq T') = \Mod(\varphi)(M' \subseteq T').
\]

\item
Given $M' \in \Mod(\cod{\varphi})$ and $M \in \Mod(\varphi)M'$ it is
obvious that $1_M \in \Mod(\varphi)1_{M'}$.

\end{itemize}
However $\Mod(\varphi)$ fails to be strict on the identities as shown
by the following counterexample.
Let $\varphi \co \{ p,q \} \pto \{ p \}$ such that 
$\DOM \varphi = \{ p \}$.
If we take $M' = \{ p \}$, $M = M'$ and $N = \{ p,q \}$ then we have
that $M \subseteq N \in \Mod(\varphi)1_{M'}$, which means that
$\Mod(\varphi)1_{M'}$ is strictly larger than $1_{\Mod(\varphi)M'} =
\{ 1_M \mid M \in \Mod(\varphi)M' \}$.
\vsp

SATISFACTION. \
The satisfaction relation of $\3/2 \PL$ is inherited from $\PL$.
The Satisfaction Condition is proved on the basis of that of $\PL$
as follows. 
Let $\varphi \co P \pto P'$, $M' \co P' \ra 2$ and $M \in
\Mod(\varphi)M'$ and $\rho\in \DOM(\Sen\varphi)$. 
Then 
$$\begin{array}{rll}
M' \models \Sen (\varphi)\rho \text{ \ if and only if \ } &
  M' \models \Sen^\PL (\varphi^0)\rho & 
  \quad \text{by definition of }\Sen(\varphi)\\ [.2em]
\text{ \ if and only if \ } & \varphi^0 ; M' \models \rho & 
  \quad \text{by the Satisfaction Condition in } \PL \text{ for } \varphi^0 \\[.2em]
\text{ \ if and only if \ } & 
  (\DOM \ \varphi \subseteq P) ; M \models \rho &  
  \quad \text{since } (\DOM \ \varphi \subseteq P) ; M = \varphi^0 ; M) \\[.2em]
\text{ \ if and only if \ } & M \models \rho & 
  \quad \text{by the Satisfaction Condition in } \PL \text{ for } \DOM \
                                               \varphi \subseteq P.
\end{array}$$
\end{rm}
\end{example}

\begin{example}[Many sorted algebra with partial morphisms of
  signatures -- $\3/2 \MSA$]\label{3/2-msa-ex}
\begin{rm}
In this example we extend the $\MSA$ institution to its $\3/2$ variant
in a way that parallels the extension of $\PL$ to $\3/2 \PL$. 
For this reason we will give only the definitions and rather skip the
arguments. 

Given $\MSA$ signatures, a \emph{partial $\MSA$-signatures morphism}
$\varphi \co (S,F) \pto (S',F')$ consists of 
\begin{itemize}

\item a partial function $\varphi^\st \co S \pto S'$, and 

\item for each $w \in (\DOM \varphi^\st)^*$ and $s \in \DOM
  \varphi^\st$ a partial function 
  $\varphi^\op_{w\ra s} \co F_{w\ra s} \pto F'_{\varphi^\st w\ra
    \varphi^\st s}$.

\end{itemize}
Given $\varphi \co (S,F) \pto (S',F')$ and $\varphi' \co (S',F') \pto
(S'',F'')$ their composition $\varphi;\varphi'$ is defined by 
\begin{itemize}

\item $(\varphi;\varphi')^\st = \varphi^\st ; \varphi'^\st$, and 

\item for each  $w \in (\DOM (\varphi;\varphi')^\st)^*$ and 
$s \in \DOM (\varphi;\varphi')^\st$: \  
$(\varphi;\varphi')^\op_{w\ra s} = 
\varphi^\op_{w\ra s} ; \varphi'^\op_{\varphi^\st w\ra \varphi^\st s}$. 

\end{itemize}
Given $\varphi,\theta \co (S,F) \pto (S',F')$, then $\varphi \leq
\theta$ if and only if 
\begin{itemize}

\item $\varphi^\st \subseteq \theta^\st$, and 

\item for each $w \in (\DOM \varphi^\st)^*$ and $s \in \DOM
  \varphi^\st$: \ $\varphi^\op_{w\ra s} \subseteq \theta^\op_{w\ra s}$.

\end{itemize}
Under these definitions the partial $\MSA$-signature morphisms form a
$\3/2$-category, which is the category of the $\3/2 \MSA$
signatures. 

Given a partial $\MSA$-signature morphism $\varphi$ we denote by $\DOM
\varphi$ the signature $(\DOM \varphi^\st, \DOM \varphi^\op)$ where  
$(\DOM \varphi^\op)_{w \to s} = \DOM \varphi^\op_{w\to s}$ and by
$\varphi^0 \co \DOM \varphi \to \cod{\varphi}$ the resulting (total)
$\MSA$-signature morphism.

For any signature $\Sigma$, 
$\Sen^{\3/2 \MSA}(\Sigma) = \Sen^{\MSA}(\Sigma)$ and for any partial
$\MSA$-signature morphism $\varphi$, $\Sen^{\3/2 \MSA}(\varphi)$ is
defined by 
\begin{itemize}

\item $\DOM \ \Sen^{\3/2 \MSA}(\varphi) = \Sen^\MSA (\DOM \varphi)$
and 

\item for each sentence $\rho \in \DOM \ \Sen^{\3/2 \MSA}(\varphi)$, 
$\Sen^{\3/2 \MSA}(\varphi)\rho = \Sen^\MSA (\varphi^0)\rho$. 

\end{itemize}
Like for $\3/2 \PL$ this yields also a \emph{strict} $\3/2$-functor. 
For any signature $\Sigma$, 
$\Mod^{\3/2 \MSA}(\Sigma) = \Mod^{\MSA}(\Sigma)$ and for any partial
$\MSA$-signature morphism $\varphi$, each $\cod{\varphi}$-model $M'$,
$\Mod^{\3/2 \MSA}(\varphi)M' = M$ is 
defined by 
\begin{itemize}

\item for each sort symbol $s$ in $\DOM \varphi$, $M_s = M'_{\varphi^\st s}$,
  and  

\item for each operation symbol $\sigma$ in $\DOM \varphi$, 
$M_\sigma = M'_{\varphi^\op \sigma}$.  

\end{itemize}
The definition on model homomorphisms is similar, we skip it here. 
Under these definitions, $\Mod^{\3/2 \MSA}$ is a lax functor. 

The satisfaction relation is inherited from $\MSA$, and the argument
for the Satisfaction Condition in $\3/2 \MSA$ is similar to that in
$\3/2 \PL$.  
\end{rm}
\end{example}

\begin{example}\label{3/2-submsa-ex}
\begin{rm}
The $\3/2 \MSA$ example can be twisted by considering less partiality
in the signature morphisms.
This can be done in several ways, in each case a different
$\3/2$-`sub-institution' of $\3/2 \MSA$ emerges. 
\begin{enumerate}

\item We constrain $\varphi^\st$ to be total functions.

\item We let $\varphi^\st$ to be partial functions but we constrain
  $\varphi^\op_{w\to s}$ to be total. 

\end{enumerate}
\end{rm}
\end{example}

\begin{example}\label{3/2-views-ex}
\begin{rm}
The pattern of Ex.~\ref{3/2-msa-ex} can be applied to the
extension of $\MSA$ that takes the `first order views' of \cite{Views}
in the role of signature morphisms.
Since first order views are more general the the $\MSA$ signature
morphisms, the resulting $\3/2$-institution based upon partial first
order views can thought as an extension of $\3/2 \MSA$.   
\end{rm}
\end{example}

\subsection{$\3/2$-institutional seeds}

So far the Examples
\ref{3/2-pl-ex}, \ref{3/2-msa-ex}, \ref{3/2-submsa-ex} and
\ref{3/2-views-ex} are based upon a pattern that can be described as
follows: 
\begin{enumerate}

\item Consider a concrete 1-institution (that may be quite common). 

\item Consider some form of partiality for its signature morphisms;
  often this can be done in several different ways (see
  Ex.~\ref{3/2-submsa-ex}).  

\item Keep the sentences and the models of the original institution,
  but based on the partiality of the signature morphisms extend the
  concepts of sentence translations and of model reducts to
  $\3/2$-institutional ones.
  The partiality of the sentence translations amounts to the fact that
  only the sentences that only involve symbols from the definition domain
  of the (partial) signature morphism can be translated.
  The relation-like aspect of the model reducts amounts to the fact
  that symbols that are outside the definition domain of the (partial)
  signature morphisms can be interpreted in several different ways in
  the models.  

\item The satisfaction relation of the resulting $\3/2$-institution is
  inherited from the original 1-institution. 

\end{enumerate}
This pattern pervades a lot of useful $\3/2$-institutions and can be
captured as a generic mathematical construction that derives
$\3/2$-institutions from 1-institutions; this will be the topic of  
Sect.~\ref{gen-3/2-institution-sec}. 
However there are significant examples of $\3/2$-institutions that fall
short off this pattern; two of them will appear in
Sections \ref{th-morphism-sec} and \ref{th-change-sec}, respectively.
 
In the following we propose a general scheme for defining
$\3/2$-institutions that on the one hand serves a technical purpose as
it projects a convenient mathematical perspective on situations of
interest, and on the other hand constitutes a framework for generating new
$\3/2$-institutions, some of them not necessarily being
partiality-based. 

\begin{definition}[$\3/2$-institutional seed]\label{institution-seed-dfn}
A \emph{$\3/2$-institutional seed} $(\Sign,\Sen,\Omega,T)$ consists of 
\begin{itemize}

\item a lax $\3/2$-functor $\Sen \co \Sign \to \Setp$ (the
  `sentence functor'), and 

\item a designated `signature' $\Omega \in |\Sign|$ and a `truth'
  function $T \co \Sen(\Omega) \to 2$. 

\end{itemize}
\end{definition}  

\begin{proposition}\label{seed2ins-prop}
Any $\3/2$-institutional seed $\mathcal{S} = (\Sign,\Sen,\Omega,T)$ extends 
canonically to a lax $\3/2$-institution 
$\I(\mathcal{S}) = (\Sign,\Sen,\Mod,\models)$ as
follows: 
\begin{itemize}

\item for each signature $\Sigma\in |\Sign|$ we let 
\[
\Mod(\Sigma) = 
\{ M \co \Sigma \to \Omega \mid \Sen(M) \text{ total} \},
\]

\item for each signature morphism $\varphi$ and each $\cod{\varphi}$-model
  $M'$ we let 
\[
\Mod(\varphi)M' = \{ M \mid \varphi;M' \leq M \},
\]

\item for each $\Sigma$-model $M$ and each $\Sigma$-sentence $\rho$ we
  let
\[
M \models \rho \text{ \ if and only if \ } 
T(\Sen(M)\rho) = 1.
\] 
\end{itemize}
\end{proposition}  

\begin{proof}
For showing the lax functoriality of $\Mod$ we consider signature
morphisms $\varphi, \varphi'$ such that $\cod{\varphi} =
\dom{\varphi'}$ and $M'' \in \Mod(\cod{\varphi'})$. 
Then 
$$\begin{array}{rl}
  \Mod(\varphi')(\Mod(\varphi)M'') = &
  \{ M\in \Mod(\varphi)M' \mid M' \in \Mod(\varphi')M''\}  \\[.2em]
  & (\text{by the definition of composition in $\3/2
    (\CAT_\mathcal{P})$})  \\[.2em]
=  & \{ M \in \Mod(\dom{\varphi}) \mid \exists M'\in
     \Mod(\cod{\varphi}) \text{ such that } \varphi;M' \leq M,
     \varphi';M'' \leq M' \}  \\[.2em]
   & (\text{by the definitions of } \Mod(\varphi), \Mod(\varphi'))  \\[.2em]
\subseteq  & \{ M \in \Mod(\dom{\varphi}) \mid \varphi;\varphi';M'' \leq M \}  \\[.2em]
   & (\text{by the monotonicity of the composition in  } \Sign)  \\[.2em]
=  & \Mod(\varphi;\varphi')M''  \\[.2em]
  & (\text{by the definition of }\Mod(\varphi;\varphi')). 
\end{array}$$

The lax functoriality of $\Mod$ on identities may be checked as
follows:  
\[
1_{\Mod(\Sigma)}(M) = \{ M \} \subseteq \{ N \co \dom{M} \to \Omega
\mid M\leq N, \ \Sen(N) \text{ total}\} = \Mod(1_\Sigma)M.
\]

For showing the Satisfaction Condition we consider a signature
morphism $\varphi$, a $\cod{\varphi}$-model $M'$, $M \in
\Mod(\varphi)M'$ and $\rho\in \DOM \ \Sen(\dom{\varphi})$. 

Since $\varphi;M' \leq M$ by the monotonicity of $\Sen$ we have that
$\Sen(\varphi;M') \subseteq \Sen(M)$. 
By the lax property of $\Sen$ it follows that
$\Sen(\varphi);\Sen(M') \subseteq \Sen(M)$. 
Since $\rho\in \DOM \ \Sen(\varphi)$ and since $\Sen(M')$ is total it
follows that $\Sen(M')(\Sen(\varphi)\rho) = \Sen(M)\rho$. 
Consequently $T(\Sen(M')(\Sen(\varphi)\rho)) = 
T(\Sen(M)\rho)$ which means 
$M' \models \Sen(\varphi)\rho \ = \ M \models \rho$. 
\end{proof}

The following two situations show that Prop.~\ref{seed2ins-prop} is
a vehicle for obtaining natural $\3/2$-institutions.

\begin{example}[Seeds for $\3/2 \PL$, $\3/2 \MSA$]
\begin{rm}
\noindent
\begin{enumerate}

\item The $\3/2 \PL$ variant without model homomorphisms arises easily as an
$\I(\mathcal{S})$ by taking $\Omega =2$ and by taking $T$ to be the
function that evaluates Boolean terms (for example $T(\neg (0 \wedge 1))
= 1$, etc.)

\item 
Even a \emph{local} variant of $\3/2 \MSA$ without model homomorphisms
such that all carrier sets of the models are subsets of a fixed
set $U$ arises as a $\I(\mathcal{S})$ by defining 
$\Omega = (S^\Omega, F^\Omega)$ by 
\begin{itemize} 

\item $S^\Omega = 2^U$, i.e. the sets of the subsets of $S$, and

\item for any $s_1, \dots, s_n, s \subseteq U$,
  $F^\Omega_{s_1 \dots s_n \ra s}$ is the set of all functions 
  $s_1 \times \dots \times s_n \to s$. 

\end{itemize}
The truth function $T$ is based upon the evaluation of $\Omega$-terms
by recursion and functional composition as follows:
\begin{itemize}

\item Any term $t$ of sort $s$ gets evaluated as an element $T(t) \in
  s$ (note here the overloading of $T$) defined by 
\[
T(\sigma(t_1,\dots,t_n)) = \sigma(T(t_1),\dots,T(t_n)).
\]

\item For any equation $t_1 = t_2$ we set $T(t_1 = t_2) = 1$ if and
  only if $T(t_1) = T(t_2)$. 

\item The evaluation function $T$ extends to composed sentence,
  in an obvious manner in the case of the Boolean connectives, and
  as follows in the case of quantifications. 
  Given an $\Omega$-sentence $(\forall x)\rho$ where $x$ is a variable
  of sort $s$, then 
  \[
  T((\forall x)\rho) = \bigwedge \{ T(\rho(a)) \mid a \in s \}
  \]
  where $\rho(a)$ denotes the $\Omega$-sentence obtained by replacing
  each occurence of $x$ in $\rho$ by $a$.  

\end{itemize}
\end{enumerate}
\end{rm}
\end{example}

Because the definition of $\3/2$-institutional seeds involves
deceptively poor data, there is a significant space for defining
relevant $\3/2$-institutions that do not fall into the pattern of
partiality of signature morphisms. 
The following example, albeit rather artificial, may give an
indication about this potential.  

\begin{example}[A seed beyond partiality]\label{nonpartial-ex}
\begin{rm}
We let 
\begin{itemize}

\item $|\Sign| = \omega$, the set of the natural numbers,

\item arrows $m \to n$ are pairs $(a,b)$ of natural numbers such that
  $a \leq n - m$, 

\item the composition of arrows $(a,b) \co m\to n$ and $(c,d)\co n \to
  p$ is $(a+c,b\vee d) \co m \to p$ \\[.2em]
  (by $b \vee d$ we denote the maximum of $b$ and $d$); note that the
  composition is well defined, it is associative and has $(0,0)$ as
  identities. 

\end{itemize}
So far this yields a category. 
Now we make this into a $\3/2$-category. 
\begin{itemize}

\item Given $(a,b), (a',b') \co m \to n$ we let $(a,b)\leq (a',b')$ if
  and only if $a=a'$ and $b' \leq b$.
  It is easy to check that this yields a partial order which is
  preserved by the compositions. 

\end{itemize}
The lax $\3/2$-functor $\Sen \co \Sign \to \Setp$ is defined as
follows: 
\begin{itemize}

\item for each $m\in \omega$, $\Sen(m) = \{ x\in \omega \mid x\leq m
  \}$, 

\item for each arrow $(a,b) \co m \to n$ in $\Sign$, 
$\DOM \ \Sen(a,b) = \{ x\in \omega \mid x\leq m, \ x+a+b \leq n \}$ and
$\Sen(a,b)(x) = x+a$ for each $x\in \DOM \ \Sen(a,b)$.

\end{itemize}
The interested reader may check the lax functoriality properties of
$\Sen$; we skip this here. 

Now any choice of $\Omega$ and $T\co \Sen(\Omega) \to 2$ completes the
definition of a $\3/2$-institutional seed.   
\end{rm}
\end{example}

\subsection{Model amalgamation in $\3/2$-institutions}

The following definition extends the crucial notion of model
amalgamation concept from 1-institutions to $\3/2$-institutions.
For the sake of simplicity of presentation, this is presented for lax
cocones of spans, the general concept for lax cocones over arbitrary
diagrams of signature morphisms being an obvious generalisation.  
Moroever all the results in this section can be presented in that more
general framework without a real additional effort. 

\begin{definition}\label{consistent-dfn}
A \emph{model for a diagram of signature morphisms} in a
$\3/2$-institution consists of a model $M_k$ for each signature
$\Sigma_k$ in the diagram such that for each signature morphism
$\varphi \co \Sigma_i \to \Sigma_j$ in the diagram we have that 
$M_i \in \Mod(\varphi)M_j$. 

The diagram is \emph{consistent} when it has at least one model. 
\end{definition} 

\begin{definition}[Model amalgamation in
  $\3/2$-institutions]\label{amalg-dfn} 
In any $\3/2$-institution, a lax cocone for a span in the
$\3/2$-category of the signature morphisms  
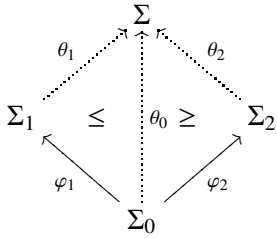
\[
\xymatrix @C-2em{
 & & \Sigma & & \\
\Sigma_1 \ar@{.>}[urr]^{\theta_1} & { \ \ \ \ \leq} & &
{ \geq \ \ \ \ } &
\Sigma_2 \ar@{.>}[ull]_{\theta_2}\\
 &  & \Sigma_0 \ar[ull]^{\varphi_1} \ar[urr]_{\varphi_2} \ar@{.>}[uu]_{\theta_0}& &
}
\]
has \emph{model amalgamation} when each model of the span admits an
unique completion to a model (called the \emph{amalgamation}) of the
lax cocone.  

When dropping the uniqueness condition, the property is called
\emph{weak model amalgamation}.  
\end{definition}

Note that when the signature morphisms involved in
Dfn.~\ref{amalg-dfn} are total (or at least when the model reducts
give singletons) we get the ordinary concept of model amalgamation for
(1-)institution theory.  
This also means that $\theta_0$ and $\Sigma_0$-model become redundant. 
In the proper $\3/2$ case their presence is necessary, this being one
of the important aspects that distinguishes the $\3/2$ case from
ordinary (1-)institution theoretic model amalgamation.   

\begin{example}\label{amg-ex}
\begin{rm}
In $\3/2 \PL$, for the diagram of Dfn.~\ref{amalg-dfn}  we consider
the signatures  
$\Sigma_0 = \{ p,p',p_1,p_2 \}$, $\Sigma_1 = \{ p,p_1,p'_1 \}$,
$\Sigma_2 = \{ p,p_2,p'_2 \}$, $\Sigma = \{ p,p',p'_1,p'_2 \}$ and let
$\varphi_1$, $\varphi_2$, $\theta_0$, $\theta_1$, $\theta_2$ be the
maximal partial inclusions.  
We prove that this cocone has model amalgamation as follows.
We assume $\{ M_k \mid k=0,1,2 \}$ a model for the span
$(\varphi_1,\varphi_2)$ and define the $\Sigma$-model $M$ by $M(p) =
M_k (p)$, $M(p'_k) = M_k (p'_k)$, $k\in 1,2$, and $M(p') = M_0 (p')$.  
It is easy to see that $M$ thus defined is the unique amalgamation of
$M_0,M_1,M_2$. 
\end{rm}
\end{example}

In ordinary institution theory the causal dependency between
pushout squares and model amalgamation squares is 
central and well known (cf. \cite{modalg,iimt,sannella-tarlecki-book},
etc.).
The following result refines this to $\3/2$-institutions in a way
intended to maximize its applicability in concrete situations.  

\begin{proposition}\label{amg-seed-prop}
For any $\3/2$-institutional seed $\mathcal{S}$ and any 1-subcategory
$\T \subseteq \Sign$ such that 
\begin{itemize}

\item $\Sen$ preserves and reflects maximality 
($\varphi$ is maximal if and only if it is $\Sen$-maximal), 

\item $\T$ contains all maximal signature morphisms, and 

\item if $\varphi \in \T$ and $\varphi \leq \varphi'$ then 
$\varphi' \in \T$,

\end{itemize}
in $\I(\mathcal{S})$ each lax $\T$-pushout of signature morphisms has
weak model amalgamation. 
\end{proposition}

\begin{proof}
We consider a lax $\T$-pushout $(\theta_0,\theta_1,\theta_2)$ for a span
$(\varphi_1,\varphi_2)$ of signature morphisms like shown in the
diagram below, and a model $\{ M_k \mid k = 0,1,2 \}$ for the span 
$(\varphi_1,\varphi_2)$. 
By the first and second assumptions this means that we have a lax
$\T$-cocone $(M_0,M_1,M_2)$ for the span $(\varphi_1,\varphi_2)$. 
By the universal property of $(\theta_0,\theta_1,\theta_2)$ there
exists an unique signature morphism $M \co \Sigma \to \Omega$ in $\T$
such that $\theta_k ; M = M_k$ for $k=0,1,2$. 
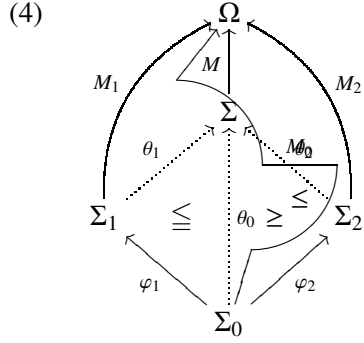
\begin{equation}\label{amg-diag-equation}
\xymatrix @C-2em{
 & & \Omega && \\
 & & \Sigma \ar[u]^{M} & & \\
\Sigma_1 \ar@{.>}[urr]^{\theta_1} \ar@/^{1.5em}/@{->}[uurr]^{M_1} 
  & { \ \ \ \ \leqq} & &
{ \geq \ \ \ \ } &
\Sigma_2 \ar@{.>}[ull]_{\theta_2} \ar@/_{1.5em}/@{->}[uull]_{M_2}\\
 &  & \Sigma_0 \ar[ull]^{\varphi_1} \ar[urr]_{\varphi_2}
 \ar@{.>}[uu]_{\theta_0} \ar `r/60pt[u] `[uuu]_{M_0}^*++{\leq} [uuu]
& &
}
\end{equation}
In order to establish that $M$ is a model we show that $M$ is maximal;
then since $\Sen$ preserves maximality it follows that $\Sen(M)$ is
total.
 
Let $M \leq N$. 
By the third assumption it follows that $N \in \T$. 
For each $k = 0,1,2$, by the monotonicity of the composition, we have
that $M_k = \theta_k ; M \leq \theta_k ; N$. 
Because $M_k$ is maximal (as a consequence of $\Sen$ reflecting
maximality) it follows that $M_k = \theta_k ; N$ for each
$k=0,1,2$. 
By the uniqueness of $M$ as a meditating arrow between lax
$\T$-cocones it follows that $M=N$.  
Hence $M$ is maximal. 
\end{proof}

One quick note on the first condition of Prop.~\ref{amg-seed-prop}
which although holds naturally in many $\3/2$-institutions of interest
(such as those from Ex.~\ref{3/2-pl-ex}, \ref{3/2-msa-ex},
\ref{3/2-submsa-ex} and \ref{3/2-views-ex}), it has to be assumed in
the abstract setup since there are concrete situations when it does
not hold (such as the $\3/2$-institution of Ex.~\ref{nonpartial-ex}
where $\Sen$ preserves maximality but does not reflect it).\vsp

The following result gives the important information that we should in
general give up expectations that weak lax cocones may involve
`non-total' signature morphisms; this will be also used to strengthen
the conclusion of Prop.~\ref{amg-seed-prop}.  

\begin{proposition}\label{total-seed-prop}
For any $\3/2$-institutional seed $\mathcal{S}$ and any 1-subcategory
$\T \subseteq \Sign$ such that 
\begin{itemize}

\item $\Sen$ is strict, and 

\item $\T$ contains all $\Sen$-maximal signature morphisms, 

\end{itemize}
for any consistent span $(\varphi_1,\varphi_2)$ of signature morphisms
in the $\3/2$-institution $\I(\mathcal{S)}$ any of 
its each weak lax $\T$-pushout cocones $(\theta_0,\theta_1,\theta_2)$
consists only of $\Sen$-maximal signature morphisms.     
\end{proposition}

\begin{proof}
The consistency of the span means that it has a lax cocone
$(M_0,M_1,M_2)$ such that each $\Sen(M_k)$ is total for $k=0,1,2$. 
By the second assumption of the proposition it follows that this is a
$\T$-cocone. 
By the weak lax $\T$-pushout property of $(\theta_0,\theta_1,\theta_2)$
there exists an $M \co \Sigma \to \Omega$ in $\T$ such that 
$\theta_k ; M = M_k$ for $k=0,1,2$ (like in diagram
\eqref{amg-diag-equation}). 
Since $\Sen$ is strict it follows that 
$\Sen(\theta_k);\Sen(M) = \Sen(M_k)$, $k=0,1,2$.
Because $\Sen(M_k)$ is total, $\Sen(\theta_k)$ must be total too. 
\end{proof}

The outstanding condition of Prop.~\ref{total-seed-prop} is that of
consistency of the span. 
Although at the abstract level the consistency of spans has to be
assumed axiomatically, in concrete situations, spans of real signature 
morphisms are very easily consistent.  
For example in $\3/2 \PL$ it is enough to consider $(M_k)p =1$,
$k=0,1,2$, for all propositional symbols $p$, and in $\3/2 \MSA$ to
consider $M_k$, $k=0,1,2$, having a fixed singleton set $\{ * \}$ as
underlying/carrier sets.
However the concept gets real substance in $\3/2$-institutions where
the signature morphisms carry more structure than the common signature 
morphisms, an important example being given by that of theory
morphisms of Sect.~\ref{th-morphism-sec} below.  

\begin{corollary}\label{amg-seed-cor}
If in addition to the hypotheses of Prop.~\ref{amg-seed-prop} we have
that $\Sen$ is strict then the conclusion of Prop.~\ref{amg-seed-prop}
is that in $\I(\mathcal{S})$ each lax $\T$-pushout of signature
morphisms has model amalgamation.  
\end{corollary}

\begin{proof}
Let us suppose that a model $\{ M_k \mid k=0,1,2 \}$ of the span
$(\varphi_1,\varphi_2)$ has two amalgamations $M$ and $N$.
In other words $\theta_k;M, \ \theta_k;N \leq M_k$ for $k=0,1,2$.  

Note that the second assumption of Prop.~\ref{total-seed-prop} is a
consequence of the assumptions of Prop.~\ref{amg-seed-prop}.
By the strictness of $\Sen$ we have that $\Sen(\theta_k;M) =
\Sen(\theta_k);\Sen(M)$ for $k=0,1,2$ and likewise for $N$. 
Since $\Sen(\theta_k)$ (by Prop.~\ref{total-seed-prop}), $\Sen(M)$,
$\Sen(N)$ (since $M,N$ are models) are total functions, it
follows that all $\Sen(\theta_k;M)$, $\Sen(\theta_k;N)$, $k=0,1,2$,
are total functions too.  
By the first assumption of Prop.~\ref{amg-seed-prop} it follows that
all $\theta_k;M$, $\theta_k;N$, $k=0,1,2$, are maximal.
Hence $\theta_k;M = \theta_k;N = M_k$, $k=0,1,2$.  
By the uniqueness part of the universal property of lax $\T$-pushouts
it follows that $M=N$. 
\end{proof}

The following corollary indicates that the result of
Cor.~\ref{amg-seed-cor} covers many concrete situations of interest. 

\begin{corollary}
In both $\3/2 \PL$ and $\3/2 \MSA$ each lax $\T$-pushout of signature
morphisms has model amalgamation in any of the following situations
for $\T$ (the latter two apply only for $\3/2 \MSA$):
\begin{enumerate}

\item all signature morphisms,

\item the total signature morphisms, 

\item the signature morphisms that are total on the sort symbols,
  i.e. $\varphi^{\st}$ are total functions, and

\item the signature morphisms that are total on the operation symbols,
  i.e. $\varphi^{\op}_{w\to s}$ are total functions. 

\end{enumerate}
\end{corollary}

\begin{proof}
Recall from Sect.~\ref{3/2-institution-sec} how $\3/2 \PL$ arises as
an $\I(\mathcal{S})$. 
In the case of $\3/2 \MSA$, although due to cardinality issues it
cannot be presented as a whole as an $\I(\mathcal{S})$, we may
consider `localised' versions that have all carriers of models
included in a fixed set $U$. 
Thus, given a span of signature morphisms an a model $\{ M_k \mid k =
0,1,2 \}$ of it, we may take $U$ to be the union of all the carrier
sets in $M_0, M_1, M_2$. 
Then the hypotheses of Prop.~\ref{amg-seed-prop} and
Cor.~\ref{amg-seed-cor} can be checked quite easily in each of the
cases for $\T$ \ listed in the statement of the corollary.  
\end{proof}

So far we have established model amalgamation for classes of lax
cocones that enjoy a universal property of a colimit. 
In the following we develop some results that may be used to extend
model amalgamation to other classes of lax cocones. 
First we need a couple of new concepts.  

\begin{definition}[Model conservativeness]
In a $\3/2$-institution a signature morphism $\varphi$ is \emph{model
  conservative} when for each $\dom{\varphi}$-model $M$ there exists a
$\cod{\varphi}$-model $M'$ such that $M \in \Mod(\varphi)M'$. 
\end{definition}

In general, in many concrete situations of interest -- $\3/2 \PL$ and
$\3/2 \MSA$ included -- a signature morphism is model conservative if
and only if it is injective (this does not exclude the possibility of
partiality).   

\begin{definition}[Model strictness]\label{model-strict-dfn}
In a $\3/2$-institution a signature morphism $\varphi$ is
\emph{model$\Mod$-strict} when for each signature morphism $\theta$ such
that $\cod{\theta} = \dom{\varphi}$ we have that 
\[
\Mod(\varphi);\Mod(\theta)=\Mod(\theta;\varphi).
\]
\end{definition} 

In general, in many concrete situations of interest -- $\3/2 \PL$ and
$\3/2 \MSA$ included -- a signature morphism is $\Mod$-strict whenever 
it is total. 
One way to see this is through the following general result. 

\begin{proposition}\label{model-strict-prop}
For any $\3/2$-institutional seed $\mathcal{S}$, any $\Sen$-maximal
signature morphism is $\Mod$-strict in the associated
$\3/2$-institution $\I(\mathcal{S})$. 
\end{proposition}

\begin{proof}
Since the other inclusion holds by the lax functoriality of $\Mod$, we
need only to prove that for each $\cod{\varphi}$-model $M''$ we have
that 
\[
\Mod(\theta;\varphi)M'' \subseteq \Mod(\theta)(\Mod(\varphi)M'').
\]
Any $M \in \Mod(\theta;\varphi)M''$ is characterised by the properties
that $\Sen(M)$ is total and that 
\begin{equation}\label{strict-equation}
\theta;\varphi;M'' \leq M.
\end{equation}
Now since $\Sen(\varphi)$ and $\Sen(M'')$ are total functions it
follows that their composition is a total function too, hence by the
lax functoriality of $\Sen$ is follows that $\Sen(\varphi;M'')$ is
a total function too. 
This means that $\varphi;M''$ is a model in $\Mod(\varphi)M''$. 
This and \eqref{strict-equation} imply that $M \in
\Mod(\theta)(\Mod(\varphi)M'')$.   
\end{proof}



\begin{proposition}\label{weak-amg-prop}
In any $\3/2$-institution, consider a lax cocone
$(\theta_0,\theta_1,\theta_2)$ of a span of signature morphisms
$(\varphi_1,\varphi_2)$ and a signature morphism $\mu$ such that 
$\cod{\theta}=\dom{\mu}$.
Then 
\begin{enumerate}

\item if the lax cocone $\theta$ has weak model amalgamation and $\mu$ is
  model conservative then the lax cocone $\theta;\mu$ has it too, and 

\item if there exists a lax cocone $\theta'$ that has weak model
  amalgamation and such that $\theta;\mu \leq \theta'$, and
  $\mu$ is $\Mod$-maximal and model $\Mod$-strict then the lax cocone 
  $\theta$ has weak model amalgamation too.   

\end{enumerate}
\end{proposition}

\begin{proof}
1.
Consider a model $\{ M_k \mid k=0,1,2 \}$ for the span
$(\varphi_1,\varphi_2)$. 
There exists a $\cod{\theta}$-model $M$ such that $M_k \in
\Mod(\theta_k)M$,  $k=0,1,2$. 
Since $\mu$ is model conservative there exists a model $M'$ such that
$M\in \Mod(\mu)M'$. 
Then for each $k\in 0,1,2$, $M_k \in \Mod(\theta_k)(\Mod(\mu)M')
\subseteq \Mod(\theta_k;\mu)M'$ (by the lax property of $\Mod$). 
Hence $M'$ is an amalgamation of $\{ M_k \mid k=0,1,2 \}$.  \vsp

2.
Consider a model $\{ M_k \mid k=0,1,2 \}$ for the span
$(\varphi_1,\varphi_2)$. 
There exists a
$\cod{\mu}$-model $M'$ such that $M_k \in \Mod(\theta'_k)M'$,
$k=0,1,2$. 
Since $\theta_k;\mu \leq \theta'_k$, $k=0,1,2$, and since $\Mod$
preserves orders, we have that 
$\Mod(\theta'_k)M' \subseteq\Mod(\theta_k;\mu)M'$,
$k=0,1,2$.    
Hence  $M_k \in \Mod(\theta_k ;\mu)M'$, $k=0,1,2$.
 
By the $\Mod$-maximality assumption we have that 
$\Mod(\mu)M' = \{ M \}$. 
By the $\Mod$-strictness assumption it follows that for each $k=0,1,2$, 
$M_k \in \Mod(\theta_k)(\Mod(\mu)M') = \Mod(\theta_k)M$. 
Hence $M$ is an amalgamation of $\{ M_k \mid k = 0,1,2 \}$. 
\end{proof}

We can combine Prop.~\ref{amg-seed-prop} and \ref{weak-amg-prop} for
getting a larger class of lax cocones enjoying weak model
amalgamation. 

\begin{corollary}\label{weak-amg-seed-cor}
Under the hypotheses of Prop.~\ref{amg-seed-prop} we consider a lax
cocone $(\theta_0,\theta_1,\theta_2)$ for a span of signature
morphisms $(\varphi_1,\varphi_2)$ and a signature morphism $\mu$ such
that $\cod{\theta}=\dom{\mu}$.
Then 
\begin{enumerate}

\item if $\theta$ is a lax $\T$-pushout and $\mu$ is model
  conservative then the lax cocone $\theta;\mu$ has weak model
  amalgamation, and   

\item if there exists a lax $\T$-pushout $\theta'$ such that
  $\theta;\mu \leq \theta'$ and $\mu$ is $\Sen$-maximal
  then the lax cocone $\theta$ has weak model amalgamation.

\end{enumerate}
\end{corollary}

\begin{proof}
While 1. is a direct consequence of  Prop.~\ref{amg-seed-prop} and
\ref{weak-amg-prop}, the argument for 2. needs a bit of elaboration.  
By Prop.~\ref{model-strict-prop} we get that $\mu$ is $\Mod$-strict. 

Now let $M$ be any $\cod{\mu}$-model.
Because $\Sen(\mu)$ and $\Sen(M)$ are total functions, by the lax
functoriality of $\Sen$ it follows that $\Sen(\mu;M)$ is a total
function too. 
Since $\Sen$ reflects maximality (one of the hypothesis of
Prop.~\ref{amg-seed-prop}) it follows that $\mu;M$ is maximal, hence
$\Mod(\mu)M = \{ \mu;M \}$. 
This shows that $\mu$ is $\Mod$-maximal. 

Now all conditions of Prop.~\ref{amg-seed-prop} and
\ref{weak-amg-prop} are fulfilled, therefore the conclusion
2. follows. 
\end{proof}

\begin{example}
\begin{rm}
The (weakened version of the) model amalgamation situation of
Ex.~\ref{amg-ex} can be obtained from Cor.~\ref{weak-amg-seed-cor}
(2.) as follows. 
\begin{itemize}

\item We set $\T$ \ to be the class of the total functions.

\item For each $k=0,1,2$ we let $\theta'_k$ to be the inclusion of
  $\Sigma_k$ into  $\{ p,p',p_1,p'_1,p_2,p'_2 \}$.  
  This is a $\T$-pushout. 

\item We let $\mu$ be the inclusion $\{ p,p',p'_1,p'_2 \} \subseteq 
\{ p,p',p_1,p'_1,p_2,p'_2 \}$. 

\end{itemize}
\end{rm}
\end{example}

\subsection{Theory morphisms in $\3/2$-institutions}\label{th-morphism-sec}

In 1-institution theory, the concept of theory morphism plays an
important role in connection to foundational works in computer science. 
It was one of the central institution theoretic concepts introduced
and studied in the seminal publication
\cite{Goguen-Burstall:Institutions-1992}.  
The mathematical foundations of conceptual blending are based on
theory morphisms since concepts are modelled as logical theories and
their translations as theory morphisms
\cite{Goguen:Algebraic-Semiotics-1999,Goguen:Mathematical-Models-of-Cognitive-Space-and-Time-2006}.  
While theories in $\3/2$-institutions are the same as theories in
1-institutions, the $\3/2$-institution theoretic concept of theory
morphism is much more subtle because of the partiality of the sentence
translations.    
In fact there are at least four ways to extend the 1-institution
concept of theory morphism to $\3/2$-institutions. 

\begin{definition}\label{th-morphism-dfn}
In a $\3/2$-institution a \emph{theory} $(\Sigma,E)$ consists of a
signature $\Sigma$ and a set $E$ of $\Sigma$-sentences 
($E \subseteq \Sen(\Sigma)$). 

Given two theories $(\Sigma,E)$ and $(\Sigma',E')$ in a
$\3/2$-institution, a signature morphism 
$\varphi \co \Sigma \to \Sigma'$ is  
\begin{itemize}

\item a \emph{pseudo-morphism of theories} when $\Sen(\varphi)E
  \subseteq E'^\bullet$,  

\item a \emph{weak morphism of theories} when $\Sen(\varphi)E^\bullet
  \subseteq E'^\bullet$,  

\item a \emph{strong morphism of theories} when for each
  $\Sigma'$-model $M'$ such that $M' \models E'$ there exists $M \in
  \Mod(\varphi)M'$ such that $M \models E$, and 

\item a \emph{ultra-strong morphism of theories} when for all
  $\Sigma'$-models $M'$ and $\Sigma$-models $M$ such that $M' \models
  E'$ and $M \in \Mod(\varphi)M'$ we have that $M \models E$.  

\end{itemize}
\end{definition}

\begin{fact}
Any weak morphism is pseudo-morphism, any strong morphism is
weak.
If $\Mod$ does not admit emptiness then any ultra-strong morphism is
strong.  
\end{fact}

In 1-institution theory the four concepts of theory morphisms of
Dfn.~\ref{th-morphism-dfn} collapse to the single established
1-institution concept of theory morphism 
(cf. \cite{Goguen-Burstall:Institutions-1992,iimt}, etc.).  
But in the realm of $\3/2$-institutions they are in general different
concepts as shown by the following very simple counterexamples: 
\begin{itemize}

\item In $\3/2 \PL$ consider $\Sigma = \{p,q \}$, $\Sigma' = \{ q \}$, 
$E = \{ p \wedge q \}$, $E' = \emptyset$.
Then $\varphi$, the maximal partial inclusion of $\Sigma$ into
$\Sigma'$ ($\DOM \varphi = \{ q \}$), is a pseudo-morphism 
$(\Sigma,E)\to (\Sigma',E')$ but it is not a weak one since 
$q \in \Sen(\varphi)E^\bullet \setminus E'^\bullet$. 

\item 
  In the quantifier-free variant of $\3/2 \MSA$ (which means sentences
  without quantifiers) consider $\Sigma$ consisting of one sort symbol
  $s$ and two constants $c$, $c'$, $\Sigma'$ consisting only of the
  sort symbol $s$ and a constant $c'$, $E = \{ \neg (c = c') \}$, and
  $E'=\emptyset$. 
  Then $\varphi$, the maximal partial inclusion of $\Sigma$ into
  $\Sigma'$, is a (trivially) weak morphism $(\Sigma,E)\to
  (\Sigma',E')$ but it is not a strong one since any singleton set
  does not admit a $\varphi$-reduct that satisfies
  $E'$.\footnote{Counterexample communicated by Daniel G\u{a}in\u{a}.}  

\item In $\3/2 \PL$ consider $\Sigma = \{p,q \}$, $\Sigma' = \{ q \}$, 
$E = \{ p \wedge q \}$, $E' = \{ q \}$.
Then $\varphi$, the maximal partial inclusion of $\Sigma$ into
$\Sigma'$ ($\DOM \varphi = \{ q \}$), is a strong morphism  
$(\Sigma,E)\to (\Sigma',E')$ but it is not an ultra-strong one. 
There exists only one model $M' \models E'$, namely $M'(q) = 1$. 
Then $M'$ has a $\varphi$-reduct $M$ such that $M \models E$ defined
by $M(p) = M(q) = 1$. 
However \emph{not any} $\varphi$-reduct of $M'$ enjoys this property,
for example $N$ such that $N(p)=0$ and $N(q) =1$. 

\end{itemize} 

In general pseudo-morphisms and do not compose and the ultra-strong
ones compose under the condition that $\Mod$ is strict rather than
(properly) lax. 
The strictness condition on $\Mod$ is a very heavy and unrealistic one
in the applications (actually unlike the strictness condition on
$\Sen$ which holds in a lot of $\3/2$-institutions of interest).
This makes both extremes, the pseudo-morphisms and the ultra-strong
morphisms, unsuitable as a $\3/2$-institutional replacement for the
1-institution theory morphisms and leaves us only with the middle
options. 
But it is not only the failure in compositionality that makes them
unsuitable, their very nature also feel inadequate as can be for
example seen by inspecting the very simple examples above. 
Pseudo-morphisms are too weak and the ultra-strong morphisms seem to 
require too much. 
The strong theory morphisms compose unconditionally, while the weak
ones compose under a certain condition that holds often in concrete
situations.  

\begin{proposition}
In any $\3/2$-institution $\I$, by inheriting the $\3/2$-categorical 
structure of $\Sign^\I$ 
\begin{itemize}

\item strong morphisms of theories yield a $\3/2$-category -- denoted
  $\Th_s^\I$, and 

\item when $\Sen$ is oplax, the weak theory morphisms yield a
  $\3/2$-category -- denoted $\Th_w^\I$. 

\end{itemize} 
\end{proposition}

\begin{proof}
The proof is based on the fact that the composition of theory
morphisms yields a theory morphism; the rest being straightforward. 
Let us consider theory morphisms $\varphi \co (\Sigma,E) \to
(\Sigma',E')$ and $\varphi' \co (\Sigma',E') \to (\Sigma'',E'')$.
 
For the `strong' case we consider $M'' \in |\Mod(\Sigma'')|$ such that
$M''\models E''$.
Then there exists $M' \in \Mod(\varphi')M''$ such that $M' \models
E'$.
It follows that there exists $M\in \Mod(\varphi)M'$ such that 
$M \models E$. 
Then by the lax property of $\Mod$ it follows that $M \in
\Mod(\varphi;\varphi')$. 

For the `weak' case we have:
$$\begin{array}{rll}
\Sen(\varphi;\varphi') E^\bullet \subseteq & 
  (\Sen(\varphi);\Sen(\varphi')) E^\bullet & 
  \quad \text{by the oplax functoriality of }\Sen \\[.2em]
  = & \Sen(\varphi') (\Sen(\varphi) E^\bullet) & \quad \text{by Lemma
                                                 \ref{pfun-lem}}  \\[.2em]
  \subseteq  & \Sen(\varphi') E'^\bullet & \quad \text{since }\varphi
                                           \text{ is theory morphism}  \\[.2em]
  \subseteq  & E''^\bullet & \quad \text{since }\varphi'
                                           \text{ is theory morphism}. 
\end{array}$$
\end{proof}

From now on whenever we encounter weak theory morphisms we tacitly
assume that $\Sen$ is oplax. 
\vsp

The constructions in the Corollaries \ref{th-institution-cor1} and
\ref{th-institution-cor2} constitute natural examples of
$\3/2$-institutions that are not based on an explicit form of
partiality of signature morphisms. 

\begin{corollary}\label{th-institution-cor1}
For any $\3/2$-institution $\I = (\Sign,\Sen,\Mod,\models)$ its
$\3/2$-category of weak/strong theory morphisms determines a
$\3/2$-institution $\I_w / \I_s$ as follows ($i$ is $w$ or $s$):   
\begin{itemize}

\item the $\3/2$-category of signatures $\Sign_i$ is $\Th_i^\I$,  

\item $\Sen_i$ is a trivial lifting of $\Sen$ to theories,
  i.e. $\Sen_i (\Sigma,E) = \Sen_i (\Sigma)$, etc., 

\item $\Mod_i (\Sigma,E)$ is the full subcategory of $\Mod(\Sigma)$
  of the $\Sigma$-models satisfying $E$, and for each theory morphism
  $\varphi \co (\Sigma,E) \to (\Sigma',E')$ and each
  $(\Sigma',E')$-model $M'$
\[
\Mod_i (\varphi)M' = \{ M \in \Mod(\varphi)M' \mid M \models E \}
\]

\item and the satisfaction relation is inherited from $\I$. 

\end{itemize}
\end{corollary}

\begin{proof}
The only interesting part of the proof is the lax functoriality of
$\Mod_{i}$, the rest being straightforward. 
We consider $\varphi\co (\Sigma,E) \to (\Sigma',E')$ and $\varphi'\co
(\Sigma',E') \to (\Sigma'',E'')$ theory morphisms.  
For any $(\Sigma'',E'')$-model $M''$ we have that 
\[
\begin{array}{rll}
\Mod_i (\varphi)(\Mod_i(\varphi')M'') = & 
\Mod_i (\varphi)\{ M' \in \Mod(\varphi')M'' \mid M' \models E' \} 
 & \text{definition of }\Mod_i\\[.2em]
= & \{ M \in \Mod(\varphi)M' \mid M' \in \Mod(\varphi')M'', M
    \models E, M' \models E' \}  
& \text{definition of }\Mod_i\\[.2em]
\subseteq & \{ M \in \Mod(\varphi)M' \mid M' \in \Mod(\varphi')M'', \ M
    \models E \} 
 & \\[.2em]
= & \{ M \in \Mod(\varphi)(\Mod(\varphi')M'') \mid  M \models E \} & \\[.2em]
\subseteq & \{ M \in \Mod(\varphi;\varphi')M'' \mid  M \models E \} &
\text{since }\Mod\text{ is lax}\\[.2em]
= & \Mod_i (\varphi;\varphi')M'' & \text{definition of }\Mod_i. 
\end{array}
\]
\end{proof}

$\I_w / \I_s$ generalise the concept of the ``institution of
theories'' from 1-institution theory \cite{iimt} to
$\3/2$-institutions.  
Note that both of them constitute examples of $\3/2$-institutions
where the model functor may naturally admit emptiness, and
this without being inherited from the base institution.
\vsp 

There is also an alternative way to complete the definition of  
$\Th^{\I}_w / \Th^{\I}_s$ to that of a $\3/2$-institution by shifting
the weight of the construction from the models side to the sentences 
side. 
However this construction is conditioned by $\I$ being a 
\emph{lax $\3/2$-institution}. 

\begin{corollary}\label{th-institution-cor2}
For any lax $\3/2$-institution $\I= (\Sign,\Sen,\Mod,\models)$ its
$\3/2$-category of weak/strong theory morphisms determines a lax
$\3/2$-institution $\I_{w'} / \I_{s'}$ as follows ($i$ is $w$ or $s$):    
\begin{itemize}

\item the $\3/2$-category of signatures $\Sign_{i'}$ is $\Th_i^\I$,  

\item $\Sen_{i'}(\Sigma,E) = E^\bullet$ and for each theory morphism
  $\varphi \co (\Sigma,E) \to (\Sigma',E')$ we let 
\begin{itemize}

\item $\DOM \ \Sen_{i'}(\varphi) = E^\bullet \cap \DOM \
  \Sen(\varphi)$, and 

\item $\Sen_{i'}(\varphi)\rho = \Sen(\varphi)\rho$ for all 
$\rho\in \DOM \ \Sen_{i'}(\varphi)$.

\end{itemize}

\item $\Mod_{i'}$ is the trivial lifting of $\Mod$, i.e. 
$\Mod_{i'} (\Sigma,E) = \Mod(\Sigma)$, etc., 

\item and the satisfaction relation is inherited from $\I$. 

\end{itemize}
\end{corollary}

\begin{proof}
The only interesting part of the proof is the lax functoriality of
$\Sen_{i'}$, the rest being straightforward. 
We consider $\varphi\co (\Sigma,E) \to (\Sigma',E')$ and $\varphi'\co
(\Sigma',E') \to (\Sigma'',E'')$ theory morphisms. 
On the one hand we have that 
\[
\begin{array}{rll}
\DOM \ \Sen_{i'}(\varphi);\Sen_{i'}(\varphi') = &
\DOM \ \Sen_{i'}(\varphi) \cap \Sen_{i'}^{-1}(\varphi)(\DOM \
                                                  \Sen_{i'}(\varphi'))
  & \\[.2em]
= &
E^\bullet \cap \DOM \ \Sen(\varphi) \cap \Sen(\varphi)^{-1}(E'^\bullet
\cap \DOM \ \Sen(\varphi')) & \text{definition of }\Sen_{i'}\\[.2em]
= & E^\bullet \cap \DOM \ \Sen(\varphi) \cap \Sen(\varphi)^{-1}(E'^\bullet)
\cap \Sen(\varphi)^{-1}(\DOM \ \Sen(\varphi')) & \\[.2em]
= &  E^\bullet \cap \DOM \ (\Sen(\varphi);\Sen(\varphi')) \cap
    \Sen(\varphi)^{-1}(E'^\bullet) & \\[.2em]
\subseteq & E^\bullet \cap \DOM \ \Sen(\varphi;\varphi') \cap
    \Sen(\varphi)^{-1}(E'^\bullet) & \Sen \text{ is lax}\\[.2em]
\subseteq  & E^\bullet \cap \DOM \ \Sen(\varphi;\varphi') \\[.2em]
= & \DOM \ \Sen_{i'}(\varphi;\varphi') & \text{definition of }\Sen_{i'}. 
\end{array}
\]
On the other hand for each 
$\rho \in \DOM \ \Sen_{i'}(\varphi);\Sen_{i'}(\varphi')$, 
\[
\Sen_{i'}(\varphi')(\Sen_{i'}(\varphi)\rho) = 
\Sen(\varphi')(\Sen(\varphi)\rho)  =
\Sen(\varphi;\varphi')\rho =
\Sen_{i'}(\varphi;\varphi')\rho. 
\] 
\end{proof}

One of the starting motivations in 1-institution theory was the
development of a general logic-independent method for the aggregation
of software modules, modelled as institutional theories
\cite{Goguen-Burstall:Institutions-1992}. 
The process of ``putting together'' -- just to use a favourite phrase
of Goguen and Burstall -- institutional theories relies on
colimits in the category of theory morphisms, an important result
being the automatic lifting of colimits from the category of signature
morphisms to that of theory morphisms (see
\cite{Goguen-Burstall:Institutions-1992,iimt,sannella-tarlecki-book}). 
The following results replicate this in the context of
$\3/2$-institutions in support of conceptual blending theory.  
The more complicated situation of colimits and theory morphisms
in $\3/2$-institutions leads to a significantly more complex situation
with respect to the lifting of colimits from signatures to theories. 

\begin{proposition}[Lifting lax cocones from signatures to
  theories]\label{lift-prop} 
Consider a span of weak/strong theory morphisms 
$\varphi_k \co (\Sigma_0,E_0) \to (\Sigma_k, E_k)$, $k=1,2$, and a lax
cocone for the underlying span of signature morphisms like shown in
the following diagram.  
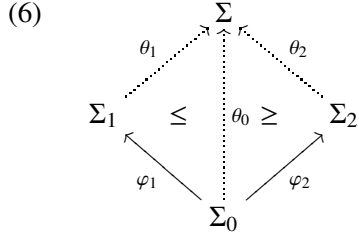
\begin{equation}\label{lax-cocone-signature-equation}
\xymatrix @C-2em{
 & & \Sigma & & \\
\Sigma_1 \ar@{.>}[urr]^{\theta_1} & { \ \ \ \ \leq} & &
{ \geq \ \ \ \ } &
\Sigma_2 \ar@{.>}[ull]_{\theta_2}\\
 &  & \Sigma_0 \ar[ull]^{\varphi_1} \ar[urr]_{\varphi_2} \ar@{.>}[uu]_{\theta_0}& &
}
\end{equation}
Then for any $E\subseteq \Sen(\Sigma)$ such
that $\bigcup_{k=0,1,2} \Sen(\gamma_k) E_k^\bullet \subseteq E^\bullet$
the following diagram displays a lax cocone of
theory morphisms
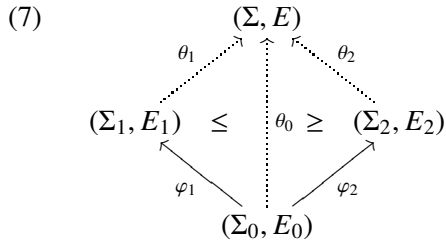
\begin{equation}\label{lax-cocone-th-equation}
\xymatrix @C-3em{
 & & (\Sigma,E) & & \\
(\Sigma_1,E_1) \ar@{.>}[urr]^{\theta_1} & { \ \ \ \ \ \leq} & &
{ \geq \ \ \ \ \ } &
(\Sigma_2,E_2) \ar@{.>}[ull]_{\theta_2}\\
 &  & (\Sigma_0,E_0) \ar[ull]^{\varphi_1} \ar[urr]_{\varphi_2} \ar@{.>}[uu]_{\theta_0}& &
}
\end{equation}
where
\begin{enumerate}

\item in the `weak' case, $\gamma_k = \theta_k, k=0,1,2$, and

\item in the `strong' case, $\gamma_k, k=0,1,2$ are any signature
  morphisms such that $\theta_k \leq \gamma_k$ and $\Sen(\gamma_k)$
  are total functions. 

\end{enumerate}
\end{proposition}

\begin{proof}
We have to only to show that $\theta_k \co (\Sigma_k,E_k) \to
(\Sigma,E)$, $k=0,1,2$ are theory morphisms.  
The `weak' case is straightforward.
For the `strong' case we consider any $M \in |\Mod(\Sigma)|$ such that
$M \models E$. 
Because $\Sen(\gamma_k)$ are total, by the Satisfaction Condition it
follows that for any $M_k \in \Mod(\gamma_k)M$, $M_k \models E_k$.
Since $\theta_k \leq \gamma_k$ by the monotonicity of $\Mod$ it
follows that  $M_k \in \Mod(\theta_k)M$. 
Hence $\theta_k$ are strong theory morphisms.
\end{proof}

\begin{corollary}[Lifting lax pushouts from signature to theories]
In the context of Prop.~\ref{lift-prop}, given a 1-subcategory 
$\T \subseteq \Sign$ let $\T_w$/$T_s$ denotes the class of weak theory
morphisms $\varphi$ such that $\varphi\in \T$. 
We further assume that
\begin{itemize}

\item $\Mod$ does not admit emptiness,

\item the lax cocone of signature morphisms is a lax $\T$-pushout,

\item $E^\bullet = 
(\bigcup_{k=0,1,2} \Sen(\gamma_k) E_k^\bullet)^\bullet$. 

\end{itemize}
Then the lax cocone of theory morphisms obtained by Prop.~\ref{lift-prop}  
\begin{itemize}

\item[--] is a lax $\T_w$-pushout when $\Sen$ is lax (therefore it is
  strict) and each signature morphism in $\T$ is $\Sen$-maximal, 

\item[--] is a lax $\T_s$-pushout when each signature morphism in $\T$ is
  $\Mod$-maximal.  

\end{itemize}
\end{corollary}

\begin{proof}
We consider a lax $\T_w$/$\T_s$-cocone $\theta'$ for the span of
weak/strong theory morphisms.  
By the lax $\T$-pushout property in $\Sign$ (the category of signature
morphisms) there exists an unique $\mu \in \T$ such that 
$\theta_k ; \mu = \theta'_k$, $k= 0,1,2$. 
It only remains to show that $\mu$ is a weak/strong theory morphism
$(\Sigma,E) \to (\Sigma',E')$, where $(\Sigma',E')$ is the vertex of
$\theta'$. 
\vsp

We first solve the weak case. 
Let us recall that in this case $\gamma_k = \theta_k$. 
For that we need the following lemma (we skip its proof):
\begin{lemma}\label{conseq-lem}
In any $\3/2$-institution such that $\Mod$ does not admit emptiness,
for any signature morphism $\varphi$ that is $\Sen$-maximal and for
any set $E$ of $\dom{\varphi}$-sentences, we have that 
\[
\Sen(\mu)E^\bullet \subseteq (\Sen(\mu)E)^\bullet.
\]
\end{lemma}
Then 
\[
\begin{array}{rll}
\Sen(\mu)E^\bullet = & \Sen(\mu)(\bigcup_{k=0,1,2} \Sen(\theta_k)
                       E_k^\bullet)^\bullet & \\[.2em]
\subseteq &  (\Sen(\mu)\bigcup_{k=0,1,2} \Sen(\theta_k) E_k^\bullet)^\bullet &
\text{by the second and third assumptions and by Lemma \ref{conseq-lem}}
  \\[.2em]
= &  (\bigcup_{k=0,1,2} \Sen(\mu)(\Sen(\theta_k)
            E_k^\bullet))^\bullet & \\[.2em]
= & (\bigcup_{k=0,1,2} \Sen(\theta_k;\mu) E_k^\bullet))^\bullet & 
\text{by the strictness assumption on }\Sen \\[.2em]
= & \bigcup_{k=0,1,2} \Sen(\theta'_k) E_k^\bullet))^\bullet & \\[.2em]
\subseteq & (E'^\bullet)^\bullet & \text{since }\theta'_k \text{ are
                                   weak theory morphisms} \\[.2em]
= & E'^\bullet. & 
\end{array}
\]
Now comes the strong case.
We consider a $\Sigma'$-model $M'$ such that $M' \models E'$.
Since $\mu, \theta_k \in \T$ \ are $\Mod$-maximal, let $M$ be the
unique model in $\Mod(\mu)M'$ and for each $k= 0,1,2$ let $M_k$ be the
unique model in $\Mod(\theta_k)M$. 
Since $\theta_k \leq \gamma_k$, by the monotonicity of $\Mod$ it
follows that $\Mod(\gamma_k)M \subseteq \Mod(\theta_k)M$.
Since $\Mod$ does not admit emptiness this means that $M_k$ is the
unique member of $\Mod(\gamma_k)M$ too. 

By the lax property of $\Mod$ and by the equalities $\theta'_k =
\theta_k;\mu$ it follows that 
\[
\Mod(\theta_k)(\Mod(\mu)M') \subseteq \Mod(\theta'_k)M'
\]
which means 
\[
\Mod(\theta_k)M \subseteq \Mod(\theta'_k)M'.
\]
By the $\Mod$-maximality assumption it follows that $\Mod(\theta'_k)M'
= \{ M_k \}$. 
Since $\theta'_k$ is a strong theory morphism $(\Sigma_k,E_k) \to
(\Sigma',E')$ we have that $M_k \models E_k$. 
By the Satisfaction Condition for $\gamma_k$ (and by keeping in mind that
$\Sen(\gamma_k)$ is total) we obtain that $M \models
\Sen(\gamma_k)E_k^\bullet$, $k=0,1,2$. 
This shows that $M \models E$. 
\end{proof}

The only apparently restrictive assumption in the applications is the
$\Sen$/$\Mod$-maximality condition on the signature morphisms in
$\T$. 
Very often $\Sen$ and $\Mod$-maximality say the same thing, namely
that the corresponding signature morphisms are total. 
However Prop.~\ref{total-seed-prop} tells us that in many situations
of interest, anyway one cannot get beyond that with lax
$\T$-pushouts. 
Although this does not constitute a real restriction in the
applications, we may also note that the weak case adds a supplementary 
technical condition to the strong case, namely that $\Sen$ is lax.  
\vsp   

\begin{proposition}[Lifting model amalgamation from signatures to
  theories]
\label{amalg-prop}
Under the framework of Prop.~\ref{lift-prop}, if 
\begin{itemize}

\item the lax cocone of signature morphisms has (weak) model
  amalgamation, and 

\item $E^\bullet = (\bigcup_{k=0,1,2} \Sen(\gamma_k)
  E_k^\bullet)^\bullet$ 

\end{itemize}
then the lax cocone of theory morphisms has (weak) model amalgamation
too. 
\end{proposition}

\begin{proof}
We treat both the `weak' and the `strong' case in one shot because
there is no essential difference between them.

Let $i\in \{ w,s \}$. 
We consider $(M_0,M_1,M_2)$ a model for the span of theory morphisms. 
According to the definition of $\Mod_i$ we have that $M_0 \in
\Mod(\varphi_k)M_k$ for $k=1,2$. 
We show that if $M$ is an amalgamation of $M_0$, $M_1$, and $M_2$ with
respect to the lax cocone of signature morphisms then it is
an amalgamation with respect to the lax cocone of theory
morphisms too.  
 
Let $k\in \{ 0,1,2 \}$. 
Since $M_k \in \Mod(\gamma_k)M$, since $M_k \models
E_k^\bullet$, by the Satisfaction Condition it follows that 
$M \models \Sen(\gamma_k)E_k^\bullet$.  
Hence $M \models \bigcup_{k=0,1,2}\Sen(\gamma_k)E_k^\bullet$.
Therefore $M \models
(\bigcup_{k=0,1,2}\Sen(\gamma_k)E_k^\bullet)^\bullet = E^\bullet$.
This completes the proof for the weak model amalgamation case.

The conclusion can be extended to the proper (non-weak) model
amalgamation case by noting (by a simple \emph{reductio ad absurdum}
argument) that the uniqueness of amalgamation at the level of
signature morphisms implies the uniqueness at the level of theory
morphisms.   
\end{proof}

\subsection{Theory changes}\label{th-change-sec}

In this section we develop an alternative concept of mapping between
theories in $\3/2$-institutions that does not resemble or generalise
the theory morphisms from 1-institution theory, but which models
software changes.  
Theory changes formalise the process of modifications in specification
or declarative programs. 
In this modelling a flat (unstructured) specification or program is
modelled by a theory. 
Modifications or changes operate at two different levels, at the
signature and the sentences level. 
The changes at the signature level are encapsulated in the respective
concept of signature morphism, while those at the sentences level are
made explicit and modelled by the partial inclusion component of the
concept of theory changes. 
This represents a marking of the part of the translated sentences that
is not touched by the change, which may consist both of deletions or
of adding sentences. 
The fact that the partial inclusion is not necessarily maximal
accounts for the possibility that sentences may be deleted and later
added back, or viceversa. 
Also we assume that the programmer is not committed to the parts that
he leaves unchanged. 
\vsp

First we develop a theory of partial inclusions.
A partial function $f \co A \pto B$ is an \emph{inclusion} when $f$
consists only of pairs of elements of the form $(a,a)$. 
It follows that $f \subseteq (A \cap B)^2$ and that $f = \{ (a,a) \mid
a \in \DOM f \}$. 
Note that, unlike in the case of total inclusions, given two sets $A$
and $B$ they may admit more than one partial inclusion between them
and in any case at least one (the empty one). 

Given $A_1,A_2 \subseteq A$, a partial function $f\co A \pto B$ and a
partial inclusion $i \co A_1 \pto A_2$ we let 
$f(i) = \{ (f^0(a),f^0(a)) \mid a \in \DOM(f), (a,a)\in i \}$. 

\begin{lemma}\label{pinc-lem}
$f(i)$ is a partial inclusion $f(A_1) \pto f(A_2)$. 
\end{lemma}

Another fact gives a functorial property for the above notation:

\begin{lemma}\label{pinc2-lem}
Given $A_1,A_2,A_3 \subseteq A$, a partial function $f\co A \pto B$ and 
partial inclusions $i_1 \co A_1 \pto A_2, i_2 \co A_2 \pto A_3$, we
have that $f(i_1;i_2) = f(i_1);f(i_2)$.  
\end{lemma}

Based on Lemmas \ref{pfun-lem} and \ref{pinc-lem} we get another
property: 

\begin{lemma}\label{pinc3-lem}
Given partial functions $f \co A \pto B$ and $g \co B \pto C$, sets
$A_1,A_2 \subseteq A$ and partial inclusion $i \co A_1 \pto A_2$ we
have $(f;g)(i)=g(f(i))$.  
\end{lemma}

\begin{definition}[Theory changes]
In any $\3/2$-institution a \emph{theory change} $(\varphi,i) \co
(\Sigma,E) \to (\Sigma',E')$ consists of:
\begin{itemize}

\item theories $(\Sigma,E)$ and $(\Sigma',E')$;

\item a signature morphism $\varphi\co \Sigma \ra \Sigma'$; and 

\item a partial inclusion $i \co \Sen(\varphi)E \pto E'$. 

\end{itemize}
\end{definition}

\begin{proposition}
For any $\3/2$-institution $\I$ with a strict sentence functor theory
changes form a $\3/2$-category as follows: 
\begin{itemize}

\item the composition of theory changes is as shown by the following
  diagram:  
\[
\xymatrix{
(\Sigma,E) \ar[r]^{(\varphi,i)} 
\ar[dr]_{\hspace{-6em}(\varphi;\theta,\Sen(\theta)(i);j)} & 
  (\Sigma',E') \ar[d]^{(\theta,j)} \\[.2em]
  & (\Sigma'',E'')
}
\]

\item the partial order on theory changes $(\Sigma,E) \to (\Sigma',E')$
is given by: 
\[
(\varphi,i) \leq (\varphi',i') \ \text{ \ if and only if \ } \
\varphi \leq \varphi' \text{ and } i \subseteq i'.
\]

\end{itemize}
\end{proposition}

\begin{proof}
The composition of theory changes is correctly defined because 
\begin{itemize}

\item by lemma \ref{pinc-lem} $\Sen(\theta)(i)$ is a partial
  inclusion $\Sen(\theta)(\Sen(\varphi)E) \pto \Sen(\theta)E'$, 

\item the composition of partial inclusions is a partial inclusion,
  hence $\Sen(\theta)(i);i'$ is a partial inclusion
  $\Sen(\theta)(\Sen(\varphi)E) \pto E''$, and 

\item by Lemma \ref{pfun-lem} and by the \emph{strict} functoriality
of $\Sen$ we have that $\Sen(\theta)(\Sen(\varphi)E) =
  \Sen(\varphi;\theta)E$. 

\end{itemize}
The partial order on theory changes is also correctly defined because
whenever $\varphi \leq \theta$ this implies $\Sen(\varphi)\subseteq
\Sen(\theta)$ which implies $\Sen(\varphi)E \subseteq \Sen(\theta)E$. 
Then $i\subseteq j$ parses as a subset relationship between subsets of
$\Sen(\theta)E\times E'$.  

The understanding of the proof of the associativity of the composition
of theory changes is helped by inspecting the following diagram:  
\[
\xymatrix @C+5em {
(\Sigma,E) \ar@{=}[d] 
\ar@/^2em/@{.>}[rr]^-{((\varphi;\varphi');\varphi'',
\Sen(\varphi'')(\Sen(\varphi')(i);i')
  \ ; \ i'')}  
\ar@{.>}[r]_-{(\varphi;\varphi',\Sen(\varphi')(i);i')}
& (\Sigma'',E'') \ar[r]_{(\varphi'',i'')}  
& (\Sigma''',E''') \ar@{=}[d] \\[.2em]
(\Sigma,E) \ar[r]^{(\varphi,i)} 
\ar@/_2em/@{.>}[rr]_-{(\varphi;(\varphi';\varphi'') , 
\Sen(\varphi';\varphi'')(i);\Sen(\varphi'')(i');i'')}
& (\Sigma',E') \ar[u]^{(\varphi',i')}  
\ar@{.>}[r]^-{(\varphi';\varphi'',\Sen(\varphi'')(i');i'')}
& (\Sigma''',E''') 
}
\]
Thus all we have to show is that 
$\Sen(\varphi'')(\Sen(\varphi')(i);i') ; i'' = 
\Sen(\varphi';\varphi'')(i);\Sen(\varphi'')(i');i''$, its proof being: 
$$\begin{array}{rll}
\Sen(\varphi'')(\Sen(\varphi')(i);i') = 
   & \Sen(\varphi'')(\Sen(\varphi')(i));\Sen(\varphi'')(i')
   & \quad \text{by Lemma \ref{pinc2-lem}} \\[.2em]
=  & (\Sen(\varphi');\Sen(\varphi''))(i);\Sen(\varphi'')(i')
   & \quad \text{by Lemma \ref{pinc3-lem}} \\[.2em]
=  & \Sen(\varphi';\varphi'')(i);\Sen(\varphi'')(i')
   & \quad \text{by the \emph{strict} functoriality of }\Sen.
\end{array}$$

For showing the preservation of partial orders by compositions we
consider only the case when $(\varphi;i) \leq (\varphi',i')$ and
$\cod{\varphi} = \cod{\varphi'} = \dom{\theta}$, the other situation
getting a similar proof. 
By the definition of composition we have that 
\begin{itemize}

\item $(\varphi,i);(\theta,j) = (\varphi;\theta, \Sen(\theta)(i);j)$,
  and 

\item $(\varphi',i');(\theta,j) = (\varphi';\theta,
  \Sen(\theta)(i');j)$. 

\end{itemize}
From the monotonicity of composition in $\Sign$ it follows that
$(\varphi,i)\leq (\varphi',i')$. 
From $i \subseteq i'$ it follows that 
$\Sen(\theta)(i) \subseteq \Sen(\theta)(i')$ and further that 
$\Sen(\theta)(i);j \subseteq \Sen(\theta)(i');j$. 
\end{proof}

The following is another example of a $\3/2$-institution that does not
fall into the partiality pattern characteristic to $\3/2 \PL$, $\3/2
\MSA$, etc. 

\begin{corollary}
For any $\3/2$-institution $\I$ with a strict sentence functor, the
$\3/2$-category of theory changes determines a $\3/2$-institution
$\I^c$ as follows: 
\begin{itemize}

\item the $\3/2$-category of signatures $\Sign^c$ is the
  $\3/2$-category of theory changes, 

\item $\Sen^c$ is a trivial lifting of $\Sen$ to theories,
  i.e. $\Sen^c (\Sigma,E) = \Sen(\Sigma)$ and $\Sen^c(\varphi,i) =
  \Sen(\varphi)$, 

\item $\Mod^c (\Sigma,E)$ is the full subcategory of $\Mod(\Sigma)$
  of the $\Sigma$-models satisfying $E$, and for each theory change
  $(\varphi,i) \co (\Sigma,E) \to (\Sigma',E')$ and each
  $(\Sigma',E')$-model $M'$
\[
\Mod^c (\varphi,i)M' = \{ M \in \Mod(\varphi)M' \mid M \models E \}
\]

\item and the satisfaction relation is inherited from $\I$. 

\end{itemize}
\end{corollary}

In what follows we investigate the possibility of modelling merges
of theory changes by pushout constructions. 
In principle, this should be based upon lifting pushouts from the
category of signatures to that of theory changes. 

\begin{proposition}\label{change-pushout-lift-prop}
In general, lax $\T$-pushouts do \emph{not} lift from the category of
signatures to that of theory changes. 
\end{proposition}

\begin{proof}
Consider a trivial (lax) $\T$-pushout of signature morphisms consisting
only of identities; let the span be $\varphi_1 = \varphi_2 = 1_\Sigma$
and the cocone be $\theta_0 = \theta_1 = \theta_2 = 1_\Sigma$. 
Let $\rho$ be a $\Sigma$-sentence and let 
$E_0 = E_1 = E_2 = \{ \rho \}$ and $i_1 = i_2 = 1_{E_0}$. 

Let us suppose that there exists a lax $\T$-pushout $(1_\Sigma,j_k),
k=0,1,2$ for the span given by $(1_\Sigma,i_1)$ and $(1_\Sigma,i_2)$. 
\begin{itemize}

\item By considering the lax cocone given by $(1_\Sigma,1_{E_0})$
  everywhere we infer that all $j_k, k=0,1,2$ are total.

\item By considering the lax cocone given by $(1_\Sigma,\emptyset)$, 
$(1_\Sigma,1_{E_0})$, $(1_\Sigma,\emptyset)$, let $(1_\Sigma,u)$ be
the unique mediating theory change.
From $(1_\Sigma,j_k);(1_\Sigma,u)=(1_\Sigma,\emptyset), k=1,2$ we
infer that $\rho \not\in \DOM \ u$.
It follows that $j_0 ; u \not= 1_{E_0}$ which is a contradiction. 
\end{itemize}
\end{proof}

By contrast to lax pushout, near pushouts lift trivially from
signatures to theory changes:

\begin{proposition}\label{change-near-pushout-lift-prop}
Given a span of theory changes $(\varphi_k,i_k) \co (\Sigma_0,E_0) \to
(\Sigma_k, E_k)$, $k=1,2$, and a near pushout for the underlying span
of signature morphisms like shown in the following diagram  
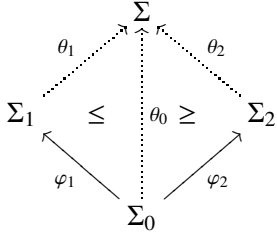
\[
\xymatrix @C-2em{
 & & \Sigma & & \\
\Sigma_1 \ar@{.>}[urr]^{\theta_1} & { \ \ \ \ \leq} & &
{ \geq \ \ \ \ } &
\Sigma_2 \ar@{.>}[ull]_{\theta_2}\\
 &  & \Sigma_0 \ar[ull]^{\varphi_1} \ar[urr]_{\varphi_2} \ar@{.>}[uu]_{\theta_0}& &
}
\]
for any $E \subseteq \Sen(\Sigma)$, $(\theta_k,\emptyset)\co
(\Sigma_k,E_k) \ra (\Sigma,E), k=0,1,2$
constitues a near pushout cocone for the given span of theory
changes. 
\end{proposition}

\begin{proof}
First, it is to establish that we have a lax cocone as
$(\varphi_k,i_k);(\theta_k,\emptyset) = (\varphi_k;\theta_k,\emptyset)
\leq (\theta_0,\emptyset)$ for $k=1,2$. 

Let $(\theta'_k,j'_k)\co (\Sigma_k,E_k) \ra (\Sigma',E'), k=0,1,2$
be a lax cocone for the given span of theory changes.
Then let $\mu$ be the maximal signature morphism such that
$\theta_k;\mu \leq \theta'_k$, $k=0,1,2$. 
We define the partial inclusion $u \co \Sen(\mu)E \pto E'$ by 
$\DOM \ u = E' \cap \Sen(\mu)E$.
Then $(\theta_k,\emptyset);(\mu,u) = 
(\theta_k;\mu,\emptyset) \leq (\theta'_k,j'_k)$, $k=0,1,2$.

Now, for any $(\mu',u')$ such that $(\theta_k,\emptyset);(\mu',u')  
\leq (\theta'_k,j'_k)$, $k=0,1,2$ we have that $\theta_k;\mu' \leq
\theta'_k$, $k=0,1,2$.
By the maximality assumption on $\mu$ it follows that $\mu'\leq \mu$. 
Since $\DOM \ u' \subseteq \Sen(\mu')E \cap E'$, since $\mu' \leq
\mu$, by the monotonicity of $\Sen$ it follows that $\DOM \ u'
\subseteq E' \cap \Sen(\mu)E  = \DOM \ u$, hence $u' \subseteq u$. 
\end{proof}

The results of Propositions \ref{change-pushout-lift-prop} and
\ref{change-near-pushout-lift-prop} tell that the established
concepts of pushouts in $\3/2$-categories cannot be used for modelling
merges of software changes.  
A new concept is needed for that.

\section{Theory blending in $\3/2$-institutions}

Now we are in the position to be able to refine Goguen's approach to
conceptual blending within the context of $\3/2$-institutions.
This appears as a stepwise process as follows:  
\begin{enumerate}

\item The input is a \emph{consistent} span of theory
  morphisms $\varphi_1, \varphi_2$ in a $\3/2$-institution $\I$, which 
  means a consistent span in $\I^t$.  

\item Then we consider an appropriate lax cocone for the underlying
  span of signature morphisms that has weak model amalgamation:  
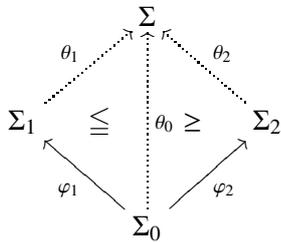
\[
\xymatrix @C-2em{
 & & \Sigma & & \\
\Sigma_1 \ar@{.>}[urr]^{\theta_1} & { \ \ \ \ \leqq} & &
{ \geq \ \ \ \ } &
\Sigma_2 \ar@{.>}[ull]_{\theta_2}\\
 &  & \Sigma_0 \ar[ull]^{\varphi_1} \ar[urr]_{\varphi_2} \ar@{.>}[uu]_{\theta_0}& &
}
\]

\item Next we lift it as in Prop.~\ref{lift-prop} to a lax cocone of
  theory morphisms: 
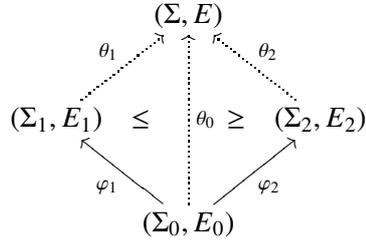
\[
\xymatrix @C-3em{
 & & (\Sigma,E) & & \\
(\Sigma_1,E_1) \ar@{.>}[urr]^{\theta_1} & { \ \ \ \ \ \leq} & &
{ \geq \ \ \ \ \ } &
(\Sigma_2,E_2) \ar@{.>}[ull]_{\theta_2}\\
 &  & (\Sigma_0,E_0) \ar[ull]^{\varphi_1} \ar[urr]_{\varphi_2} \ar@{.>}[uu]_{\theta_0}& &
}
\] 
By virtue of Prop.~\ref{amalg-prop} it follows that we obtain a oplax
cocone of theory morphisms also enjoying weak model amalgamation.
Since we started from a consistent span of theory morphisms, it
follows that the vertex of the blending cocone -- the new theory
$(\Sigma,E)$ -- is consistent. 

\end{enumerate}
This is a very general scheme that has a number of parameters.
\begin{itemize}

\item A choice of an appropriate $\3/2$-institution for modelling the
  respective concepts as theories, and their translations by theory
  morphisms. 

\item What is an `appropriate' lax cocone for the underlying span of
  signature morphisms is a challenging issue that seems to be
  difficult to answer at the general level; perhaps seeking for a
  \emph{precise} answer at a general level does not even make sense. 
  Some consider that the near pushout solution proposed by Goguen
  \cite{Goguen:Algebraic-Semiotics-1999} may be too permisive.
  Though what should be indisputable is the weak amalgamation property
  for the lax cocone. 

\end{itemize}





\section*{References}

\bibliographystyle{plain}
\bibliography{/Users/diacon/TEX/tex,ERC2016references}

\end{document}